\documentclass[10pt]{amsart}

\usepackage{amsthm}
\usepackage{amssymb}
\usepackage{amsmath}
\usepackage{easywncy}

\numberwithin{equation}{section}
\newtheorem{thm}{Theorem}[section]
\newtheorem{prop}[thm]{Proposition}
\newtheorem{lem}[thm]{Lemma}
\newtheorem{cor}[thm]{Corollary}

\theoremstyle{definition}

\newtheorem{defn}[thm]{Definition}
\newtheorem{rem}[thm]{Remark}
\newtheorem{example}[thm]{Example}

\newcommand{\be}{\mathbf{e}}
\newcommand{\bk}{\mathbf{k}}
\newcommand{\bl}{\mathbf{l}}
\newcommand{\bm}{\mathbf{m}}

\newcommand{\shp}{\mathcyr{sh}}

\begin{document}

\title[A $q$-analogue of symmetric multiple zeta value]
{A $q$-analogue of symmetric multiple zeta value}
\author{Yoshihiro Takeyama}
\address{Department of Mathematics,
  Institute of Pure and Applied Sciences,
  University of Tsukuba, Tsukuba, Ibaraki 305-8571, Japan}
\email{takeyama@math.tsukuba.ac.jp}
\thanks{This work was supported by JSPS KAKENHI Grant Number 22K03243. \\
{}\quad
This is a pre-print of an article published in The Ramanujan Journal.
The final authenticated version is available online at:
\texttt{https://doi.org/10.1007/s11139-023-00755-9}.
}

\begin{abstract}
  We construct a $q$-analogue of
  truncated version of symmetric multiple zeta values
  which satisfies the double shuffle relation.
  Using it, we define a $q$-analogue of symmetric multiple zeta values
  and see that it satisfies many of the same relations as
  symmetric multiple zeta values,
  which are the inverse relation and a part of
  the double shuffle relation and the Ohno-type relation.
\end{abstract}
%%%%%%%%%%%%%%%%%%%%%%%%%%%%%%%%%%%%%
\maketitle

\setcounter{section}{0}
\setcounter{equation}{0}

%%%%%%%%%%%%%%%%%%%%%%%%%%%%%%%%%%%%%%%%%%%%%%%%%%%%%%%%%%%%%%%%%%%%%%%%%%%%%%%

\section{Introduction}

The \textit{multiple zeta value} (MZV) is the real value defined by
\begin{align}
  \zeta(\bk)=\sum_{0<m_{1}<\cdots <m_{r}}
  \frac{1}{m_{1}^{k_{1}} \cdots m_{r}^{k_{r}}}
  \label{eq:intro-MZV}
\end{align}
for a tuple of positive integers $\bk=(k_{1}, \ldots , k_{r})$ with $k_{r}\ge 2$.
We set $\zeta(\varnothing)=1$ and regard it as a MZV.
We denote by $\mathcal{Z}$ the $\mathbb{Q}$-linear subspace of $\mathbb{R}$
spanned by the MZVs.
It is known that $\mathcal{Z}$ forms a $\mathbb{Q}$-algebra
with respect to the usual multiplication on $\mathbb{R}$.

Kaneko and Zagier introduced two kinds of variants of MZVs called
finite multiple zeta values and symmetric multiple zeta values
(see, e.g., \cite{Kaneko}).
The symmetric multiple zeta value (SMZV) is defined as an element of the quotient
$\mathcal{Z}/\zeta(2)\mathcal{Z}$.
The purpose of this paper is to construct a $q$-analogue
of the SMZVs which shares many of the relations among them.

The SMZV is defined as follows.
Although the infinite sum \eqref{eq:intro-MZV} diverges if $k_{r}=1$,
we have two kinds of its regularization,
which are called the harmonic regularized MZV $\zeta^{\ast}(\bk)$
and the shuffle regularized MZV $\zeta^{\shp}(\bk)$
(see Section \ref{subsec:SMZV} for the details).
If $k_{r}\ge 2$, they are equal to the MZV $\zeta(\bk)$.
Using them we set
\begin{align*}
  \zeta^{\mathcal{S}, \bullet}(k_{1}, \ldots , k_{r})=
  \sum_{i=0}^{r}(-1)^{k_{i+1}+\cdots +k_{r}}
  \zeta^{\bullet}(k_{1}, k_{2}, \ldots , k_{i})
  \zeta^{\bullet}(k_{r}, k_{r-1}, \ldots , k_{i+1})
\end{align*}
for $\bullet \in \{\ast, \shp\}$.
It is known that the difference $\zeta^{\ast}(\bk)-\zeta^{\shp}(\bk)$ belongs to
the ideal $\zeta(2)\mathcal{Z}$ for any tuple $\bk$ of positive integers.
The SMZV $\zeta^{\mathcal{S}}(\bk)$ is defined by
\begin{align*}
  \zeta^{\mathcal{S}}(\bk)=\zeta^{\mathcal{S}, \bullet}(\bk) \qquad
  \hbox{mod} \,\, \zeta(2)\mathcal{Z}
\end{align*}
as an element of the quotient $\mathcal{Z}/\zeta(2)\mathcal{Z}$.

Now we fix a complex parameter $q$ satisfying $0<|q|<1$
and define the $q$-integer $[n]$ for $n\ge 1$ by
$[n]=(1-q^{n})/(1-q)$.
There are various models of a $q$-analogue of the MZV
(see, e.g., \cite[Chapter 12]{ZhaoBook}).
Most of them are of the following form:
\begin{align}
  \sum_{0<m_{1}<\cdots <m_{r}}
  \frac{P_{1}(q^{m_{1}}) \cdots P_{r}(q^{m_{r}})}
  {[m_{1}]^{k_{1}} \cdots [m_{r}]^{k_{r}}},
  \label{eq:intro-qMZV}
\end{align}
where $P_{j}(x)$ is a polynomial
whose degree is less than or equal to $k_{j}$ for $1\le j \le r$.
If $P_{r}(0)=0$, the infinite sum \eqref{eq:intro-qMZV} is absolutely convergent.
In this paper we call a value of the form \eqref{eq:intro-qMZV}
a $q$-analogue of the MZV ($q$MZV) without specifying the model.

To construct a $q$-analogue of SMZVs,
it would be natural to consider two kinds of regularizations of
the sum \eqref{eq:intro-qMZV} with $P_{r}(0)\not=0$
which turn into the harmonic regularized MZV and the shuffle regularized MZV
in the limit as $q \to 1$.
However, there does not seem to exist any standard definition
of a shuffle regularization of the sum \eqref{eq:intro-qMZV}.

To avoid the difficulty, we construct a $q$-analogue of SMZVs
in a different approach.
In \cite{OSY}, Ono, Seki and Yamamoto introduced
two kinds of truncations of the
$t$-adic symmetric multiple zeta value
($t$-adic SMZV),
which is an element of the formal power series ring
$(\mathcal{Z}/\zeta(2)\mathcal{Z})[[t]]$
whose constant term is equal to the SMZV.
As a corollary, we have
the following expression for $\zeta^{\mathcal{S}, \bullet}(\bk)$.
Set
\begin{align}
  \zeta_{M}^{\mathcal{S}, \ast}(k_{1}, \ldots , k_{r})
   & =\sum_{i=0}^{r}(-1)^{k_{i+1}+\cdots +k_{r}}
  \sum_{\substack{0<m_{1}<\cdots <m_{i}<M        \\ 0<m_{r}<\cdots <m_{i+1}<M}}
  \frac{1}{m_{1}^{k_{1}} \cdots m_{r}^{k_{r}}},
  \label{eq:intro-DS1}                           \\
  \zeta_{M}^{\mathcal{S}, \shp}(k_{1}, \ldots , k_{r})
   & =\sum_{i=0}^{r}(-1)^{k_{i+1}+\cdots +k_{r}}
  \sum_{\substack{0<m_{1}<\cdots <m_{i}          \\ 0<m_{r}<\cdots <m_{i+1} \\ m_{i}+m_{i+1}<M}}
  \frac{1}{m_{1}^{k_{1}} \cdots m_{r}^{k_{r}}}.
  \label{eq:intro-DS2}
\end{align}
Then we have
$\zeta^{\mathcal{S}, \bullet}(\bk)=\lim_{M \to \infty}\zeta_{M}^{\mathcal{S}, \bullet}(\bk)$
for $\bullet \in \{\ast, \shp\}$ \cite[Corollary 2.7]{OSY}.
Note that we do not need any regularization here.
Hence we could follow the above procedure to construct
a $q$-analogue of SMZVs.
It is the main theme of this paper.

One important example of the relations among SMZVs
is the double shuffle relation
(see Theorem \ref{thm:DS-SMZV} below) due to Kaneko and Zagier.
Jarossay proved that the $t$-adic SMZVs satisfy
a generalization of the double shuffle relation
\cite{Jar}.
The point is that
the truncated version of the $t$-adic SMZVs
also satisfies them \cite[Theorem 1.6]{OSY}.
As a corollary we see that
the truncated version of SMZVs \eqref{eq:intro-DS1} and \eqref{eq:intro-DS2}
satisfies the same double shuffle relation as SMZVs.

In this paper we first construct a $q$-analogue of the
truncated SMZVs \eqref{eq:intro-DS1} and \eqref{eq:intro-DS2}
which satisfies the double shuffle relation
(Proposition \ref{prop:qS-harmonic} and Proposition \ref{prop:shuffle-rel-q-truncated}).
Second, we define a $q$-analogue of SMZVs
as a limit of the above truncated version.
As SMZVs are defined to be the elements of the quotient
$\mathcal{Z}/\zeta(2)\mathcal{Z}$,
our $q$-analogue of SMZVs is an element of the quotient of the space of
the $q$MZVs by a sum of two ideals $\mathcal{N}_{q}+\mathcal{P}_{q}$.
Roughly speaking, $\mathcal{N}_{q}$ is the ideal generated by
the $q$MZVs which turn into zero in the limit as $q\to 1-0$,
and $\mathcal{P}_{q}$ is the ideal generated by the values
$\sum_{m>0}q^{2km}/[m]^{2k}$ with $k\ge 1$,
which turns into $\zeta(2k)$ in the limit as $q \to 1-0$.
Hence the ideal $\mathcal{N}_{q}+\mathcal{P}_{q}$ could be regarded as a
$q$-analogue of to the ideal $\zeta(2)\mathcal{Z}$.

Unlike the truncated SMZVs, the double shuffle relation of
the truncated $q$SMZVs does not imply that of the $q$SMZVs.
It is because not all of the truncated $q$SMZVs converge
in the limit as $M \to \infty$
and the space spanned by them which converge is not closed
under the shuffle product (see Example \ref{ex:shuffle-not-closed} below).
Hence our $q$-analogue of SMZVs satisfies only a part of
the double shuffle relation of SMZVs.
At this stage the author does not know how to overcome this point.

The paper is organized as follows.
In Section \ref{sec:SMZV} we review on SMZV and
its truncated version and the relations of them
including the double shuffle relation.
Section \ref{sec:qMZV} gives some preliminaries
on $q$MZVs.
In Section \ref{sec:qSMZV-truncated}
we define a $q$-analogue of the truncated SMZVs
and prove the double shuffle relation.
The proof is quite similar to that for
the truncated SMZVs given in \cite{OSY}.
In Section \ref{sec:qSMZV}
we construct a $q$-analogue of SMZVs.
In Section \ref{sec:relations}
we see that our $q$-analogue shares many of the relations
with SMZVs, which are the reversal relation and
a part of the double shuffle relation and the Ohno-type relation
due to Oyama \cite{Oyama}.
Additionally two appendices follow.
In Appendix \ref{sec:app-limit}
we discuss the asymptotic behavior of $q$MZVs
in the limit as $q \to 1-0$ to prove Proposition \ref{prop:limit-of-qMZV} below.
Appendix \ref{sec:app-proofs} provides proofs
of technical propositions.

Here we give notation used throughout.
For a tuple of non-negative integers $\bk=(k_{1},  \ldots , k_{r})$,
we define its \textit{weight} $\mathrm{wt}(\bk)$ and \textit{depth} $\mathrm{dep}(\bk)$ by
$\mathrm{wt}(\bk)=\sum_{i=1}^{r}k_{i}$ and $\mathrm{dep}(\bk)=r$, respectively.
We call a tuple of positive integers an \textit{index}.
We regard the empty set $\varnothing$ as an index whose weight and depth are zero.
An index $\bk=(k_{1}, \ldots , k_{r})$ is said to be \textit{admissible} if
$\bk=\varnothing$, or $r \ge 1$ and $k_{r}\ge 2$.
We denote the set of indices (resp. admissible indices) by $I$ (resp. $I_{0}$).
The  reversal $\overline{\bk}$ of an index $\bk=(k_{1}, k_{2}, \ldots , k_{r})$ is defined
by $\overline{\bk}=(k_{r}, \ldots , k_{2}, k_{1})$.
For the empty index we set $\overline{\varnothing}=\varnothing$.

In this paper we often make use of generating functions in proofs.
Then we use the operations defined below without mention.
Suppose that $\mathfrak{A}$ is a unital algebra over a commutative ring $\mathcal{C}$.
Then we extend the addition $+$ and the multiplication $\cdot$ on $\mathfrak{A}$ to
the formal power series ring $\mathfrak{A}[[X]]$
by $f(X)+g(X)=\sum_{j \ge 0}(a_{j}+b_{j})X^{j}$ and
$f(X) \cdot g(X)=\sum_{j, l\ge 0}(a_{j}\cdot b_{l})X^{j+l}$
for $f(X)=\sum_{j\ge 0}a_{j}X^{j}$ and $g(X)=\sum_{l\ge 0}b_{l}X^{l}$
with $a_{j}, b_{l} \in \mathfrak{A}$, respectively.
Similarly, we extend a $\mathcal{C}$-linear map $\varphi: \mathfrak{A} \rightarrow \mathfrak{B}$
to the $\mathcal{C}$-linear map $\varphi: \mathfrak{A}[[X]] \rightarrow \mathfrak{B}[[X]]$ by
$\varphi(f(X))=\sum_{j\ge 0}\varphi(a_{j})X^{j}$ for
$f(X)=\sum_{j\ge 0}a_{j}X^{j}$.
We adopt the above convention for the formal power series ring
of several variables.

%%%%%%%%%%%%%%%%%%%%%%%%%%%%%%%%%%%%%%%%%%%%%%%%%%%%%%%%%%%%%%%%%%%%%%%%%%%%%%%

\section{Symmetric multiple zeta value}\label{sec:SMZV}

\subsection{Multiple zeta value}

For an index $\bk=(k_{1}, \ldots , k_{r})$ and a positive integer $M$,
we define the \textit{truncated multiple zeta value} $\zeta_{M}(\bk)$ by
\begin{align*}
  \zeta_{M}(\bk)=\sum_{0<m_{1}<\cdots <m_{r}<M}
  \frac{1}{m_{1}^{k_{1}} \cdots m_{r}^{k_{r}}}
\end{align*}
if $\bk$ is not empty, and $\zeta_{M}(\varnothing)=1$ for the empty index.
If $1\le M\le \mathrm{dep}(\bk)$, we set $\zeta_{M}(\bk)=0$.

For an admissible index $\bk$, we define the
\textit{multiple zeta value} (MZV) $\zeta(\bk)$ by
$\zeta(\bk)=\lim_{M \to \infty}\zeta_{M}(\bk)$.
If $\bk=(k_{1}, \ldots , k_{r})$ is a non-empty admissible index,
we have
\begin{align*}
  \zeta(\bk)=\sum_{0<m_{1}<\cdots <m_{r}}
  \frac{1}{m_{1}^{k_{1}} \cdots m_{r}^{k_{r}}},
\end{align*}
which converges since $k_{r} \ge 2$.

Let $\mathfrak{h}=\mathbb{Q}\langle x, y \rangle$ be the non-commutative polynomial ring
of two variables $x$ and $y$ over $\mathbb{Q}$.
We set $z_{k}=yx^{k-1}$ for $k \ge 1$, and
$z_{\bk}=z_{k_{1}} \cdots z_{k_{r}}$
for a non-empty index $\bk=(k_{1}, \ldots , k_{r})$.
For the empty index we set $z_{\varnothing}=1$.

Set $\mathfrak{h}^{1}=\mathbb{Q}+y\mathfrak{h}$.
It is a $\mathbb{Q}$-subalgebra of $\mathfrak{h}$,
which can be identified with the non-commutative polynomial ring over $\mathbb{Q}$
with the set of variables $\{z_{k} \mid k\ge 1\}$.
We also set $\mathfrak{h}^{0}=\mathbb{Q}+y\mathfrak{h}x$,
which is a $\mathbb{Q}$-submodule of  $\mathfrak{h}^{1}$
with a basis $\{ z_{\bk} \mid \bk \in I_{0} \}$.

The \textit{harmonic product} $\ast$ on $\mathfrak{h}^{1}$ is the
$\mathbb{Q}$-bilinear map
$\ast:  \mathfrak{h}^{1} \times \mathfrak{h}^{1} \longrightarrow \mathfrak{h}^{1}$
defined by the following properties:
\begin{enumerate}
  \item For any $w \in \mathfrak{h}^{1}$, it holds that $w \ast 1=w$ and $1\ast w=w$.
  \item For any $w, w'\in \mathfrak{h}^{1}$ and $k, l \ge 1$,
        it holds that $(w z_{k})\ast(w' z_{l})=(w\ast w' z_{l})z_{k}+(w z_{k}\ast w')z_{l}+
          (w\ast w')z_{k+l}$.
\end{enumerate}
The \textit{shuffle product}
$\shp: \mathfrak{h} \times \mathfrak{h} \longrightarrow \mathfrak{h}$
is similarly defined by the following properties and $\mathbb{Q}$-linearity:
\begin{enumerate}
  \item For any $w \in \mathfrak{h}$, it holds that $w \,\shp\, 1=w$ and $1\,\shp\, w=w$.
  \item For any $w, w'\in \mathfrak{h}$ and $u, v \in \{x, y \}$,
        it holds that $(w u)\,\shp\,(w' v)=(w\,\shp\, w' v)u+(w u\,\shp\, w')v$.
\end{enumerate}
Then $\mathfrak{h}^{1}$ (resp. $\mathfrak{h}$) becomes a commutative $\mathbb{Q}$-algebra
with respect to the harmonic (resp. shuffle) product, which we denote by
$\mathfrak{h}^{1}_{\ast}$ (resp. $\mathfrak{h}_{\shp}$).
We see that $\mathfrak{h}^{0}$ is a $\mathbb{Q}$-subalgebra
of $\mathfrak{h}^{1}_{\ast}$ with respect to the harmonic product.
We denote it by $\mathfrak{h}^{0}_{\ast}$.
We also see that $\mathfrak{h}^{1}$ and $\mathfrak{h}^{0}$ are
$\mathbb{Q}$-subalgebras of
$\mathfrak{h}_{\shp}$ with respect to the shuffle product.
We denote them by $\mathfrak{h}^{1}_{\shp}$ and $\mathfrak{h}^{0}_{\shp}$,
respectively.

For a positive integer $M$, we define the $\mathbb{Q}$-linear map
$Z_{M}: \mathfrak{h}^{1} \to \mathbb{Q}$
by $Z_{M}(z_{\bk})=\zeta_{M}(\bk)$ for an index $\bk$.
Similarly, we define the $\mathbb{Q}$-linear map
$Z: \mathfrak{h}^{0} \to \mathbb{R}$ by $Z(z_{\bk})=\zeta(\bk)$,
for an admissible index $\bk$.

\begin{prop}
  \begin{enumerate}
    \item For any $w, w' \in \mathfrak{h}^{1}$ and $M \ge 1$, it holds that
          \begin{align*}
            Z_{M}(w \ast w')=Z_{M}(w)Z_{M}(w').
          \end{align*}
          Hence, if $w$ and $w'$ belong to $\mathfrak{h}^{0}$,
          we have
          \begin{align}
            Z(w \ast w')=Z(w)Z(w').
            \label{eq:harmonic-rel}
          \end{align}
    \item For any $w, w' \in \mathfrak{h}^{0}$, it holds that
          \begin{align}
            Z(w \,\shp\, w')=Z(w)Z(w').
            \label{eq:shuffle-rel}
          \end{align}
  \end{enumerate}
\end{prop}

The relations \eqref{eq:harmonic-rel} and \eqref{eq:shuffle-rel} are called
the finite double shuffle relation.
They imply that the map $Z$ is a $\mathbb{Q}$-algebra homomorphism to $\mathbb{R}$
with respect to both the harmonic and shuffle product.
We denote the image of $Z$ by $\mathcal{Z}$, which is a $\mathbb{Q}$-subalgebra of $\mathbb{R}$.

\subsection{Symmetric multiple zeta value}\label{subsec:SMZV}

It is known that $\mathfrak{h}^{1}_{\bullet} \simeq \mathfrak{h}^{0}_{\bullet}[y]$
for $\bullet \in \{\ast, \shp\}$ (see \cite{Hoffman, FreeLie}).
Hence, for $\bullet \in \{\ast, \shp\}$,
there uniquely exists the $\mathbb{Q}$-algebra homomorphism
$Z^{\bullet}$ from $\mathfrak{h}^{1}_{\bullet}$ to the polynomial ring $\mathcal{Z}[T]$
such that $Z^{\bullet}(y)=T$.
For an index $\bk$, we define
$\zeta^{\ast}(\bk)=Z^{\ast}(z_{\bk})|_{T=0}$ and
$\zeta^{\shp}(\bk)=Z^{\shp}(z_{\bk})|_{T=0}$.
It is known that
\begin{align}
  \sum_{s \ge 0}X^{s}\zeta^{\ast}(\bk, \underbrace{1, \ldots , 1}_{s})=
  \exp{(\sum_{n \ge 2}\frac{(-1)^{n-1}}{n}\zeta(n)X^{n})}
  \sum_{s\ge 0}X^{s}\zeta^{\shp}(\bk, \underbrace{1, \ldots , 1}_{s})
  \label{eq:comparison-zeta}
\end{align}
in the formal power series ring $\mathcal{Z}[[X]]$
for any index $\bk$ \cite[Theorem 1 and Proposition 10]{IKZ}.

For an index $\bk=(k_{1}, \ldots , k_{r})$ and $\bullet \in \{\ast, \shp\}$, we set
\begin{align*}
  \zeta^{\mathcal{S}, \bullet}(\bk)=\sum_{i=0}^{r}
  (-1)^{k_{i+1}+\cdots +k_{r}}\zeta^{\bullet}(k_{1}, k_{2}, \ldots , k_{i})
  \zeta^{\bullet}(k_{r}, k_{r-1}, \ldots , k_{i+1}).
\end{align*}
{}From \eqref{eq:comparison-zeta}, we see that
the difference $\zeta^{\mathcal{S},\ast}(\bk)-\zeta^{\mathcal{S}, \shp}(\bk)$
belongs to the ideal
$\zeta(2)\mathcal{Z}=\pi^{2}\mathcal{Z}$ of $\mathcal{Z}$.
The \textit{symmetric MZV} (SMZV) $\zeta_{\mathcal{S}}(\bk)$ is defined by
\begin{align*}
  \zeta^{\mathcal{S}}(\bk)=\zeta^{\mathcal{S}, \bullet}(\bk) \,\, \hbox{mod} \,\,
  \pi^{2}\mathcal{Z}
\end{align*}
as an element of $\mathcal{Z}/\zeta(2)\mathcal{Z}$.

\subsection{Truncated symmetric multiple zeta value}

We define the $\mathbb{Q}$-algebra anti-automorphism $\psi$ on
$\mathfrak{h}^{1} \simeq \mathbb{Q}\langle z_{1}, z_{2}, \ldots \rangle$ by
$\psi(z_{k})=(-1)^{k}z_{k}$ for $k\ge 1$.
For $\bullet \in \{\ast, \shp\}$,
we define the $\mathbb{Q}$-linear map
$w^{\mathcal{S}, \bullet}: \mathfrak{h}^{1} \to \mathfrak{h}^{1}$ by
$w^{\mathcal{S}, \bullet}(1)=1$ and
\begin{align*}
  w^{\mathcal{S}, \bullet}(u_{1} \cdots u_{r})=\sum_{i=0}^{r}
  u_{1} \cdots u_{i} \bullet \psi(u_{i+1}\cdots u_{r})
\end{align*}
for $r \ge 1$ and $u_{1}, \ldots , u_{r} \in \{z_{k}\}_{k \ge 1}$.

For a positive integer $M$ and $\bullet \in \{\ast, \shp\}$,
we define the $\mathbb{Q}$-linear map
$Z^{\mathcal{S}, \bullet}_{M}: \mathfrak{h}^{1} \to \mathbb{Q}$  by
$Z^{\mathcal{S}, \bullet}_{M}=Z_{M} \circ w_{\mathcal{S}}^{\bullet}$.
For an index $\bk=(k_{1}, \ldots , k_{r})$ and a positive integer $M$,
the value $Z^{\mathcal{S}, \bullet}_{M}(z_{\bk})$ $(\bullet\in\{\ast,\shp\})$
is expressed as follows:
\begin{align*}
  Z_{M}^{\mathcal{S}, \ast}(z_{\bk})
   & =\sum_{i=0}^{r}(-1)^{k_{i+1}+\cdots +k_{r}}
  \sum_{\substack{0<m_{1}<\cdots <m_{i}<M        \\ 0<m_{r}<\cdots <m_{i+1}<M}}
  \frac{1}{m_{1}^{k_{1}} \cdots m_{r}^{k_{r}}},  \\
  Z_{M}^{\mathcal{S}, \shp}(z_{\bk})
   & =\sum_{i=0}^{r}(-1)^{k_{i+1}+\cdots +k_{r}}
  \sum_{\substack{0<m_{1}<\cdots <m_{i}          \\ 0<m_{r}<\cdots <m_{i+1} \\ m_{i}+m_{i+1}<M}}
  \frac{1}{m_{1}^{k_{1}} \cdots m_{r}^{k_{r}}},
\end{align*}
where $m_{0}$ and $m_{r+1}$ in the condition $m_{i}+m_{i+1}<M$
are set equal to zero in the latter formula.
We call the above values the \textit{truncated symmetric multiple zeta value}.

\begin{thm}[\cite{Jar, OSY}]\label{thm:OSY}
  \begin{enumerate}
    \item For any $w, w' \in \mathfrak{h}^{1}$ and $M \ge 1$, it holds that
          \begin{align*}
            Z_{M}^{\mathcal{S}, \ast}(w \ast w')=Z_{M}^{\mathcal{S}, \ast}(w)Z_{M}^{\mathcal{S}, \ast}(w'), \qquad
            Z_{M}^{\mathcal{S}, \shp}(w \,\shp\, w')=Z_{M}^{\mathcal{S}, \shp}(w \psi(w')).
          \end{align*}
    \item It holds that $w^{\mathcal{S}, \bullet}(\mathfrak{h}^{1}) \subset \mathfrak{h}^{0}$ for
          $\bullet \in \{\ast, \shp\}$. Hence the limit $\lim_{M \to \infty}Z_{M}^{\mathcal{S}, \bullet}(w)$
          converges for any $w \in \mathfrak{h}^{1}$ and $\bullet \in \{\ast, \shp\}$.
    \item For any index $\bk$ and $\bullet \in \{\ast, \shp\}$, it holds that
          \begin{align*}
            \lim_{M\to \infty}Z_{M}^{\mathcal{S}, \bullet}(z_{\bk})=\zeta^{\mathcal{S}, \bullet}(\bk).
          \end{align*}
  \end{enumerate}
\end{thm}

\subsection{Relations of SMZVs}

{}From the definition of $\zeta^{\mathcal{S}, \bullet}$,
we obtain the following equality,
which is called the reversal relation:
\begin{thm}
  For any index $\bk$, it holds that
  \begin{align*}
    \zeta^{\mathcal{S}}(\overline{\bk})=(-1)^{\mathrm{wt}(\bk)}
    \zeta^{\mathcal{S}}(\bk).
  \end{align*}
\end{thm}

For indices $\bk, \bl$ and $\bm$, we define the non-negative integer
$d_{\bk, \bl}^{\, \bullet, \bm} \, (\bullet\in\{\ast, \shp\})$ by
\begin{align}
  z_{\bk}\bullet z_{\bl}=\sum_{\bm}d_{\bk, \bl}^{\, \bullet, \bm}z_{\bm}.
  \label{eq:def-d}
\end{align}
The following relation is called the double shuffle relation of
SMZVs, which can be obtained from Theorem \ref{thm:OSY}.

\begin{thm}\label{thm:DS-SMZV}
  For any index $\bk$ and $\bl$, it holds that
  \begin{align*}
     &
    \zeta^{\mathcal{S}}(\bk)\zeta^{\mathcal{S}}(\bl)=
    \sum_{\bm}d_{\bk,\bl}^{\, \ast, \bm}\zeta^{\mathcal{S}}(\bm), \\
     &
    (-1)^{\mathrm{wt}(\bl)}
    \zeta^{\mathcal{S}}(\bk, \overline{\bl})=
    \sum_{\bm}d_{\bk,\bl}^{\, \shp, \bm}\zeta^{\mathcal{S}}(\bm).
  \end{align*}
\end{thm}

In \cite{Oyama}, Oyama proved that
Theorem \ref{thm:DS-SMZV} and the identity
$\zeta^{\mathcal{S}}(k, \ldots , k)=0$ for any $k \ge 1$
imply the Ohno-type relation.
To write down it, we define the Hoffman dual $\bk^{\vee}$ of an index $\bk$.
We define the $\mathbb{Q}$-algebra automorphism $\tau$ on $\mathfrak{h}$
by $\tau(x)=y$ and $\tau(y)=x$.
For a non-empty index $\bk$, the monomial $z_{\bk}$ is written in the form
$z_{\bk}=y w$ with a monomial $w \in \mathfrak{h}$.
Then the Hoffman dual index of $\bk$ is defined to be the index $\bk^{\vee}$
satisfying $z_{\bk^{\vee}}=y\tau(w)$.
For example, if $\bk=(3, 1, 2, 1)$, we have
$z_{\bk}=yx^{2}y^{2}xy$ and $z_{\bk^{\vee}}=y\tau(x^{2}y^{2}xy)=
  y^{3}x^{2}yx=z_{1}^{2}z_{3}z_{2}$,
hence $\bk^{\vee}=(1, 1, 3, 2)$.
For the empty index we set $\varnothing^{\vee}=\varnothing$.

\begin{thm}[Ohno-type relation \cite{Oyama}]
  For any index $\bk$ and any non-negative integer $m$,
  it holds that
  \begin{align*}
    \sum_{\substack{\be \in (\mathbb{Z}_{\ge 0})^{r} \\ \mathrm{wt}(\be)=m}}
    \zeta^{\mathcal{S}}(\bk+\be)=
    \sum_{\substack{\be \in (\mathbb{Z}_{\ge 0})^{s} \\ \mathrm{wt}(\be)=m}}
    \zeta^{\mathcal{S}}((\bk^{\vee}+\be)^{\vee}),
  \end{align*}
  where $r=\mathrm{dep}(\bk)$ and $s=\mathrm{dep}(\bk^{\vee})$.
\end{thm}

%%%%%%%%%%%%%%%%%%%%%%%%%%%%%%%%%%%%%%%%%%%%%%%%%%%%%%%%%%%%%%%%%%%%%%%%%%%%%%%

\section{A $q$-analogue of multiple zeta value}\label{sec:qMZV}

\subsection{Algebraic formulation}

Set $\mathcal{C}=\mathbb{Q}[\hbar]$, where $\hbar$ is a formal variable.
Let $\mathfrak{H}=\mathcal{C}\langle a, b \rangle$ be
the non-commutative polynomial ring
of two variables $a$ and $b$ over $\mathcal{C}$.
For $k \ge 1$ we set
\begin{align*}
  g_{k}=ba^{k}, \qquad e_{k}=b(a+\hbar)a^{k-1}.
\end{align*}
They are related to each other as
\begin{align*}
  g_{k}= (-\hbar)^{k-1}g_{1}+\sum_{j=2}^{k}(-\hbar)^{k-j}e_{j}
\end{align*}
for $k \ge 1$, and
\begin{align}
  e_{k}=g_{k}+\hbar g_{k-1}
  \label{eq:e-to-g}
\end{align}
for $k \ge 2$.
Note that $e_{1}-g_{1}=\hbar b$.

We set
\begin{align*}
  A=\{\hbar b\} \cup \{ba^{k} \mid k \ge 1\}=\{e_{1}-g_{1}\} \cup \{g_{k} \mid k\ge 1\},
\end{align*}
which is an algebraically independent set,
and denote by $\mathcal{C}\langle A \rangle$
the $\mathcal{C}$-subalgebra of $\mathfrak{H}$ generated by $1$ and $A$.
The \textit{depth} of a monomial $u_{1} \cdots u_{r} \, (u_{1}, \ldots , u_{r} \in A)$
is defined to be $r$.

We define the $\mathcal{C}$-submodule $\widehat{\mathfrak{H}^{0}}$
of $\mathcal{C}\langle A \rangle$ by
\begin{align*}
  \widehat{\mathfrak{H}^{0}}=\mathcal{C}+\sum_{k\ge 1}\mathcal{C}\langle A \rangle g_{k}.
\end{align*}
Then the set consisting of the elements
\begin{align}
  (e_{1}-g_{1})^{\alpha_{1}} g_{\beta_{1}+1}
  \cdots  (e_{1}-g_{1})^{\alpha_{r}} g_{\beta_{r}+1}
  \label{eq:basis-monomial}
\end{align}
with $r \ge 0$ and $\alpha_{1}, \ldots , \alpha_{r}, \beta_{1}, \ldots, \beta_{r}\ge 0$
forms a free basis of $\widehat{\mathfrak{H}^{0}}$.

Let $q$ be a complex parameter satisfying $0<|q|<1$.
We endow the complex number field $\mathbb{C}$ with $\mathcal{C}$-module structure
such that $\hbar$ acts as multiplication by $1-q$.

We denote by $\mathfrak{z}$
the $\mathcal{C}$-submodule of $\mathcal{C}\langle A \rangle$
spanned by the set $A$.
For a positive integer $m$, we define the $\mathcal{C}$-linear map
$F_{q}(m; \cdot): \mathfrak{z} \longrightarrow \mathbb{C}$ by
\begin{align}
  F_{q}(m; e_{1}-g_{1})=1-q, \qquad
  F_{q}(m; g_{k})=\frac{q^{km}}{[m]^{k}}
  \label{eq:def-Iq1}
\end{align}
for $k \ge 1$, where $[m]$ is the $q$-integer
\begin{align*}
  [m]=\frac{1-q^{m}}{1-q}=1+q+\cdots+q^{m-1}.
\end{align*}
Then, from $e_{1}=(e_{1}-g_{1})+g_{1}$ and \eqref{eq:e-to-g} for $k \ge 2$, we have
\begin{align}
  F_{q}(m; e_{k})=\frac{q^{(k-1)m}}{[m]^{k}}
  \label{eq:def-Iq2}
\end{align}
for any $k \ge 1$ because $(1-q)[m]+q^{m}=1$.

For a positive integer $M$, we define the $\mathcal{C}$-linear map
$Z_{q, M}: \mathcal{C}\langle A \rangle \longrightarrow \mathbb{C}$
by $Z_{q, M}(1)=1$ and
\begin{align*}
  Z_{q, M}(u_{1}\cdots u_{r})=\sum_{0<m_{1}<\cdots <m_{r}<M}
  \prod_{i=1}^{r}F_{q}(m_{i}; u_{i})
\end{align*}
for $u_{1}, \ldots , u_{r} \in A$.
Then we see that, if $w \in \widehat{\mathfrak{H}^{0}}$,
$Z_{q, M}(w)$ converges in the limit as $M \to \infty$.
Hence we can define the $\mathcal{C}$-linear map
$Z_{q}: \widehat{\mathfrak{H}^{0}} \longrightarrow \mathbb{C}$ by
\begin{align*}
  Z_{q}(w)=\lim_{M \to \infty}Z_{q, M}(w)
\end{align*}
for $w \in \widehat{\mathfrak{H}^{0}}$.
In this paper we call the value $Z_{q}(w)$ with  $w \in \widehat{\mathfrak{H}^{0}}$
a \textit{$q$-analogue of MZV} ($q$MZV).

\begin{rem}
  There are various models of $q$-analogue of MZV (see \cite[Chapter 12]{ZhaoBook}).
  We can represent them in the form of $Z_{q}(w)$ with
  some $w \in \widehat{\mathfrak{H}^{0}}$.
  For example, from \eqref{eq:def-Iq1} and \eqref{eq:def-Iq2}, we see that
  \begin{align*}
     &
    Z_{q}(g_{k_{1}} \cdots g_{k_{r}})=\sum_{0<m_{1}<\cdots <m_{r}}
    \frac{q^{k_{1}m_{1}+\cdots +k_{r}m_{r}}}
    {[m_{1}]^{k_{1}} \cdots [m_{r}]^{k_{r}}}, \\
     &
    Z_{q}(e_{k_{1}} \cdots e_{k_{r}})=\sum_{0<m_{1}<\cdots <m_{r}}
    \frac{q^{(k_{1}-1)m_{1}+\cdots +(k_{r}-1)m_{r}}}
    {[m_{1}]^{k_{1}} \cdots [m_{r}]^{k_{r}}},
  \end{align*}
  which are called the Schlesinger-Zudilin model and
  the Bradley-Zhao model, respectively.
\end{rem}

\subsection{Double shuffle relation of $q$MZVs}

The harmonic product and the shuffle product associated with the $q$MZV
are defined as follows.

First, we define the symmetric $\mathcal{C}$-bilinear map
$\circ_{\hbar}: \mathfrak{z}\times \mathfrak{z} \to \mathfrak{z}$ by
\begin{align*}
  (e_{1}-g_{1})\circ_{\hbar}(e_{1}-g_{1})=\hbar (e_{1}-g_{1}), \quad
  (e_{1}-g_{1})\circ_{\hbar} g_{k}=\hbar g_{k}, \quad
  g_{k}\circ_{\hbar} g_{l}=g_{k+l}
\end{align*}
for $k, l\ge 1$.
Then we see that
\begin{align}
  F_{q}(m; u \circ_{\hbar} v)=F_{q}(m; u)  F_{q}(m; v)
  \label{eq:harmonic-I}
\end{align}
for any $u, v \in \mathfrak{z}$ and $m\ge 1$.
The harmonic product $\ast_{\hbar}$ is
the $\mathcal{C}$-bilinear binary operation on
$\mathcal{C}\langle A \rangle$
uniquely defined by the following properties:
\begin{enumerate}
  \item For any $w \in \mathcal{C}\langle A \rangle$,
        it holds that $w \ast_{\hbar} 1=w$ and $1\ast_{\hbar} w=w$.
  \item For any $w, w'\in \mathcal{C}\langle A \rangle$ and $u, v \in A$,
        it holds that $(w u)\ast_{\hbar}(w' v)=(w\ast_{\hbar} w' v)u+(wu\ast_{\hbar} w')v+
          (w\ast_{\hbar} w')(u\circ_{\hbar}v)$.
\end{enumerate}
The harmonic product $\ast_{\hbar}$ on $\mathcal{C}\langle A \rangle$
is commutative and associative.

\begin{prop}\label{prop:q-harmonic}
  The $\mathcal{C}$-submodule $\widehat{\mathfrak{H}^{0}}$
  of $\mathcal{C}\langle A \rangle$ is closed
  under the harmonic product $\ast_{\hbar}$, and it holds that
  $Z_{q, M}(w \ast_{\hbar} w')=Z_{q, M}(w)Z_{q, M}(w')$ for any
  $w, w' \in  \mathcal{C}\langle A \rangle$ and $M \ge 1$.
  Therefore, we have $Z_{q}(w \ast_{\hbar} w')=Z_{q}(w)Z_{q}(w')$ for any
  $w, w' \in  \widehat{\mathfrak{H}^{0}}$.
\end{prop}

Next, we define the shuffle product.
Consider the $\mathbb{Q}$-linear right action of $\mathfrak{H}$
on a $\mathbb{C}$-valued function $f(t)$ defined by
\begin{align*}
  (f\hbar)(t)=(1-q)f(t), \qquad
  (fa)(t)=(1-q)\sum_{j=1}^{\infty}f(q^{j}t), \qquad
  (fb)(t)=\frac{t}{1-t}f(t).
\end{align*}
Then, for any function $f$ and $g$, it holds that
\begin{align}
  fa \cdot ga=(fa\cdot g+f\cdot ga+(1-q)f\cdot g)a, \quad
  fb \cdot g=f\cdot gb=(f \cdot g)b
  \label{eq:ab-rep}
\end{align}
if all terms are well-defined,
where $\cdot$ denotes the usual multiplication of functions.
Motivated by the above relations we define
the shuffle product on $\mathfrak{H}$
as the $\mathcal{C}$-bilinear binary operation
uniquely defined by the following properties:
\begin{enumerate}
  \item For any $w \in \mathfrak{H}$,
        it holds that $w \,\shp_{\hbar}\, 1=w$ and $1\,\shp_{\hbar}\,w=w$.
  \item For any $w, w'\in \mathfrak{H}$, it holds that
        \begin{align*}
           &
          wa\,\shp_{\hbar}\,w'a=(wa\,\shp_{\hbar}\, w'+w\,\shp_{\hbar}\, w'a+\hbar w\,\shp_{\hbar}\, w')a, \\
           &
          wb\,\shp_{\hbar}\, w'=w\,\shp_{\hbar}\, w'b=(w\,\shp_{\hbar}\, w')b.
        \end{align*}
\end{enumerate}
The shuffle product is commutative and associative.
Note that
\begin{align}
  (w(e_{1}-g_{1}))\,\shp_{\hbar}\,w'=w\,\shp_{\hbar}\,(w'(e_{1}-g_{1}))=
  (w\,\shp_{\hbar}\,w')(e_{1}-g_{1})
  \label{eq:e1g1-central}
\end{align}
for any $w, w' \in \mathfrak{H}$ since $e_{1}-g_{1}=\hbar b$.

\begin{prop}\label{prop:q-shuffle}
  The $\mathcal{C}$-submodules $\mathcal{C} \langle A \rangle$ and
  $\widehat{\mathfrak{H}^{0}}$ of $\mathfrak{H}$ are closed
  under the shuffle product $\shp_{\hbar}$.
  For any $u_{1}, \ldots , u_{r}, v_{1}, \ldots , v_{s} \in \mathfrak{z}$
  and $M \ge 1$, it holds that
  \begin{align}
    Z_{q, M}(u_{1} \cdots u_{r} \,\shp_{\hbar}\, v_{1} \cdots v_{s})=
    \sum_{\substack{0<m_{1}<\cdots <m_{r} \\ 0<n_{1}<\cdots <n_{s} \\ m_{r}+n_{s}<M}}
    \prod_{i=1}^{r}F_{q}(m_{i}; u_{i})
    \prod_{j=1}^{s}F_{q}(n_{j}; v_{j}).
    \label{eq:truncated-shuffle}
  \end{align}
  Therefore, we have
  $Z_{q}(w \,\shp_{\hbar}\, w')=Z_{q}(w)Z_{q}(w')$ for any
  $w, w' \in  \widehat{\mathfrak{H}^{0}}$.
\end{prop}

\begin{proof}
  Note that $g_{k+1}=g_{k}a$ for $k \ge 1$.
  For $w, w' \in \mathcal{C} \langle A \rangle$ and $k, l \ge 1$, it holds that
  \begin{align}
    wg_{1}\,\shp_{\hbar}\,w'g_{1}     & =
    (w\,\shp_{\hbar}\,w'g_{1}+wg_{1}\,\shp_{\hbar}\,w'+(w\,\shp_{\hbar}\,w')(e_{1}-g_{1}))g_{1}
    \label{eq:shp-g1g1},                  \\
    wg_{1}\,\shp_{\hbar}\,w'g_{k+1}   & =
    (w\,\shp_{\hbar}\,w'g_{k+1})g_{1}+(wg_{1}\,\shp_{\hbar}\,w'g_{k})a+
    \hbar(w\,\shp_{\hbar}\,w'g_{k})g_{1}, \\
    wg_{k+1}\,\shp_{\hbar}\,w'g_{l+1} & =
    (wg_{k}\,\shp_{\hbar}\,w'g_{l+1}+wg_{k+1}\,\shp_{\hbar}\,w'g_{l}+
    \hbar wg_{k}\,\shp_{\hbar}\,w'g_{l})a.
  \end{align}
  These formulas imply that
  \begin{align}
    \mathcal{C} \langle A \rangle g_{k}\,\shp_{\hbar}\,
    \mathcal{C} \langle A \rangle g_{l} \subset
    \sum_{\min{(k, l)}\le j \le k+l-1}\mathcal{C} \langle A \rangle g_{j}
    \label{eq:shuffle-inclusion}
  \end{align}
  for $k, l \ge 1$.
    {}From \eqref{eq:e1g1-central} and \eqref{eq:shuffle-inclusion},
  we see that $\mathcal{C} \langle A \rangle$ and
  $\widehat{\mathfrak{H}^{0}}$ are closed under the shuffle product.

  Denote by $\mathcal{F}$ the $\mathbb{C}$-vector space of holomorphic functions
  on the unit disc.
  We endow $\mathcal{F}$ with $\mathcal{C}$-module structure such that
  $\hbar$ acts as multiplication by $1-q$.
  Let $\mathbf{1}(t)=1$ be the constant function.
  We define the $\mathcal{C}$-linear map
  $L_{q}(\cdot): \mathcal{C} \langle A \rangle \longrightarrow \mathcal{F}$ by
  \begin{align*}
    L_{q}(u_{1} \cdots u_{r})(t)=((\mathbf{1})u_{1}\cdots u_{r})(t),
  \end{align*}
  where $u_{1}, \ldots , u_{r} \in A$.
  We see that
  \begin{align*}
    L_{q}(u_{1} \cdots u_{r})(t)=\sum_{0<m_{1}<\cdots <m_{r}}
    t^{m_{r}}\prod_{i=1}^{r}F_{q}(m_{i}; u_{i})
  \end{align*}
  for $u_{1}, \ldots , u_{r} \in \mathfrak{z}$.
  Hence, $Z_{q, M}(w)$ is equal to the sum of the coefficients of $t^{m}$ in $L_{q}(w)(t)$
  over $0\le m <M$ for $w \in \mathcal{C}\langle A \rangle$.
  By \eqref{eq:ab-rep}, we have
  $L_{q}(w\,\shp_{\hbar}\, w')=L_{q}(w)L_{q}(w')$ for any
  $w, w' \in \mathcal{C} \langle A \rangle$.
  Thus we obtain \eqref{eq:truncated-shuffle}.
\end{proof}

As stated in Proposition \ref{prop:q-harmonic} and Proposition \ref{prop:q-shuffle},
we have
\begin{align}
  Z_{q}(w \ast w')=Z_{q}(w)Z_{q}(w'), \qquad
  Z_{q}(w \,\shp_{\hbar}\, w')=Z_{q}(w)Z_{q}(w')
  \label{eq:DS-qMZV}
\end{align}
for any $w, w' \in \widehat{\mathfrak{H}^{0}}$.
We call them the double shuffle relation for the $q$MZVs.
We denote by $\widehat{\mathcal{Z}_{q}}$ the image of the map
$Z_{q}: \widehat{\mathfrak{H}^{0}} \rightarrow \mathbb{C}$.
Then $\widehat{\mathcal{Z}_{q}}$ is a $\mathcal{C}$-subalgebra of $\mathbb{C}$.

\subsection{Limit of $q$MZVs as $q \to 1-0$}

We denote by $\mathfrak{n}_{0}$ the $\mathcal{C}$-submodule
of $\widehat{\mathfrak{H}^{0}}$ spanned by
the elements of the form \eqref{eq:basis-monomial} with
$r \ge 1$ and $\alpha_{s}, \beta_{t}\ge 1$ for some $1\le s\le t\le r$.
We define the $\mathcal{C}$-modules $\mathfrak{n}$ and
$\mathfrak{H}^{0}$ by
\begin{align*}
  \mathfrak{n}=\mathfrak{n}_{0}+\hbar \,\widehat{\mathfrak{H}^{0}}, \qquad
  \mathfrak{H}^{0}=\mathcal{C}+
  \sum_{k \ge 2}\mathcal{C} \langle A \rangle g_{k}+\mathfrak{n}.
\end{align*}
For a non-empty index $\bk=(k_{1}, \ldots , k_{r})$, we set
\begin{align*}
  g_{\bk}=g_{k_{1}}\cdots g_{k_{r}}.
\end{align*}
For the empty index, we set $g_{\varnothing}=1$.
Note that the quotient
$\mathfrak{H}^{0}/\mathfrak{n}$
is a free $\mathbb{Q}$-module which has a basis
$\{ g_{\bk} \mid \bk \in I_{0} \}$.

\begin{prop}\label{prop:limit-of-qMZV}
  Here we consider a limit as $q \to 1-0$ with $q$ being real.
  \begin{enumerate}
    \item If $\bk$ is an admissible index, it holds that $\lim_{q \to 1-0}Z_{q}(g_{\bk})=\zeta(\bk)$.
    \item For any $w \in \mathfrak{n}$, it holds that $\lim_{q \to 1-0}Z_{q}(w)=0$.
  \end{enumerate}
\end{prop}

See Appendix \ref{sec:app-limit} for the proof.

We define the unital $\mathbb{Q}$-algebra homomorphism
$\iota: \mathcal{C} \langle A \rangle \rightarrow \mathfrak{h}^{1}$ by
\begin{align}
  \iota(\hbar)=0, \quad
  \iota(e_{1}-g_{1})=0, \quad
  \iota(g_{k})=z_{k}
  \label{eq:def-iota}
\end{align}
for $k \ge 1$.
Then we see that
the restriction of $\iota$ to $\mathfrak{H}^{0}$ is a surjection onto
$\mathfrak{h}^{0}$ and
its kernel is equal to $\mathfrak{n}$.
Therefore, from Proposition \ref{prop:limit-of-qMZV}, we obtain the following corollary.

\begin{cor}\label{cor:qMZV-lim}
  For any $w \in \mathfrak{H}^{0}$, it holds that
  \begin{align}
    \lim_{q \to 1-0}Z_{q}(w)=Z(\iota(w)).
    \label{eq:limit-of-qMZV}
  \end{align}
\end{cor}

\begin{rem}
  The equality \eqref{eq:limit-of-qMZV} does not necessarily hold
  for $w \in \widehat{\mathfrak{H}^{0}}\setminus \mathfrak{H}^{0}$ such that
  the limit of $Z_{q}(w)$ as $q \to 1-0$ converges.
  For example, we see that
  \begin{align*}
    Z_{q}((e_{1}-g_{1})g_{1}) & =(1-q)^{2}\sum_{0<m<n}\frac{q^{n}}{1-q^{n}}=
    (1-q)^{2}\sum_{0<n, l}(n-1)q^{nl}                                             \\
                              & =(1-q)^{2}\sum_{0<l}\frac{q^{2l}}{(1-q^{l})^{2}}=
    Z_{q}(g_{2}),
  \end{align*}
  which is a special case of the resummation identity
  \begin{align*}
    Z_{q}((e_{1}-g_{1})^{\alpha_{1}}g_{\beta_{1}+1} \cdots (e_{1}-g_{1})^{\alpha_{r}}g_{\beta_{r}+1})=
    Z_{q}((e_{1}-g_{1})^{\beta_{r}}g_{\alpha_{r}+1} \cdots (e_{1}-g_{1})^{\beta_{1}}g_{\alpha_{1}+1})
  \end{align*}
  for $\alpha_{1}, \ldots , \alpha_{r}, \beta_{1}, \ldots , \beta_{r} \ge 0$
  \cite[Theorem 4]{T1}.
  Hence, the limit of $Z_{q}((e_{1}-g_{1})g_{1})$ as $q \to 1-0$ is equal to $\zeta(2)$.
  However, we have $\iota((e_{1}-g_{1})g_{1})=0$.
\end{rem}

\subsection{Restoration of finite double shuffle relation of MZVs}

\begin{prop}\label{prop:n-ideal}
  The $\mathcal{C}$-module $\mathfrak{n}$ is an ideal of $\widehat{\mathfrak{H}^{0}}$
  with respect to both the harmonic product and the shuffle product.
\end{prop}

\begin{proof}
  It suffices to show that $\mathfrak{n}_{0} \ast_{\hbar} \widehat{\mathfrak{H}^{0}} \subset \mathfrak{n}$
  and  $\mathfrak{n}_{0} \,\shp_{\hbar}\, \widehat{\mathfrak{H}^{0}} \subset \mathfrak{n}$.
  The former follows from the definition of the harmonic product.
  The latter follows from \eqref{eq:e1g1-central} and \eqref{eq:shuffle-inclusion}.
\end{proof}

\begin{cor}\label{cor:n-ideal}
  The $\mathcal{C}$-module $\mathfrak{H}^{0}$ is closed under both
  the harmonic product and the shuffle product.
\end{cor}

\begin{proof}
  {}From $g_{k}\circ_{\hbar}g_{l}=g_{k+l}$,
  we see that $\sum_{k \ge 2}\mathcal{C}\langle A \rangle g_{k}$ is closed
  under the harmonic product.
  It is also closed under the shuffle product because of
  \eqref{eq:shuffle-inclusion}.
  Hence Proposition \ref{prop:n-ideal} implies that $\mathfrak{H}^{0}$
  is closed.
\end{proof}

We set
\begin{align*}
  \mathcal{Z}_{q}=Z_{q}(\mathfrak{H}^{0}).
\end{align*}
Corollary \ref{cor:n-ideal} implies that
$\mathcal{Z}_{q}$ is a $\mathcal{C}$-subalgebra of $\widehat{\mathcal{Z}_{q}}$.

\begin{prop}\label{prop:iota-prod}
  For any $w, w' \in \mathcal{C}\langle A \rangle$ and $\bullet \in \{\ast, \shp\}$,
  it holds that
  \begin{align}
    \iota(w \bullet_{\hbar} w')=\iota(w) \bullet \iota(w').
    \label{eq:iota-prod}
  \end{align}
\end{prop}

\begin{proof}
  We may assume that $w$ and $w'$ are monomials in $A$.
  If $w=1$ or $w'=1$, the desired equality \eqref{eq:iota-prod} is trivial.
  Now we proceed the proof by induction on the sum of the depth of $w$ and $w'$.

  Since $\iota((e_{1}-g_{1})\circ_{\hbar} u)=0$ for any $u \in \mathfrak{z}$
  and $\iota(g_{k} \circ_{\hbar} g_{l})=z_{k+l}$ for $k, l \ge 1$,
  we see that \eqref{eq:iota-prod} holds for $\bullet=\ast$.

  We consider the case of $\bullet=\shp$.
  If $w$ or $w'$ is the form of $u(e_{1}-g_{1})$ with
  a monomial $u$ of $A$,
  we see that \eqref{eq:iota-prod} from $\iota(e_{1}-g_{1})=0$ and \eqref{eq:e1g1-central}.
  Hence it suffices to prove the case where $w=ug_{k}$ and $w'=vg_{l}$ with
  monomials $u, v$ in $A$
  and $k, l \ge 1$.
  For that purpose we use the following formula.
  Note that, in the formal power series ring $\mathfrak{H}[[X]]$, we have
  \begin{align*}
    \sum_{k=1}^{\infty}g_{k}X^{k-1}=g_{1}\frac{1}{1-aX}.
  \end{align*}

  \begin{lem}\label{lem:shuffle-prod}
    For any $u, v \in \mathfrak{H}$, we have
    \begin{align*}
       &
      ug_{1}\frac{1}{1-aX} \,\shp_{\hbar}\, vg_{1}\frac{1}{1-aY} \\
       & =\left\{
      (1+\hbar X)\left(ug_{1}\frac{1}{1-aX}\,\shp_{\hbar}\,v\right)+
      (1+\hbar Y)\left(u\,\shp_{\hbar}\,vg_{1}\frac{1}{1-aY}\right)+
      (u\,\shp_{\hbar}\,v)(e_{1}-g_{1}) \right\}
      \\
       & \times
      g_{1}\frac{1}{1-a(X+Y+\hbar XY)}.
    \end{align*}
  \end{lem}

  See Appendix \ref{subsec:app-shpUV} for the proof.
    {}From the above equality, we see that
  \begin{align*}
    \iota(ug_{1}\frac{1}{1-aX} \,\shp_{\hbar}\, vg_{1}\frac{1}{1-aY})=
    \iota(
    ug_{1}\frac{1}{1-aX}\,\shp_{\hbar}\,v+
    u\,\shp_{\hbar}\,vg_{1}\frac{1}{1-aY}
    )z_{1}\frac{1}{1-x(X+Y)}
  \end{align*}
  in the formal power series ring $\mathfrak{h}[[X, Y]]$.
  The induction hypothesis implies that the right hand side is equal to
  \begin{align*}
    \left\{
    \iota(u)z_{1}\frac{1}{1-xX}\,\shp\,\iota(v)+
    \iota(u)\,\shp\,\iota(v)z_{1}\frac{1}{1-xY}
    \right\}
    z_{1}\frac{1}{1-(X+Y)x}.
  \end{align*}
  We see that it is equal to
  \begin{align*}
    \iota(u)z_{1}\frac{1}{1-xX}\,\shp\,\iota(v)z_{1}\frac{1}{1-xY}=
    \iota(ug_{1}\frac{1}{1-aX})\,\shp\,
    \iota(vg_{1}\frac{1}{1-aY})
  \end{align*}
  in the same way as the proof of Lemma \ref{lem:shuffle-prod}.
  This completes the proof of Proposition \ref{prop:iota-prod} for $\bullet=\shp$.
\end{proof}

Using Corollary \ref{cor:n-ideal} and Proposition \ref{prop:iota-prod},
we restore the finite double shuffle relation of MZVs from \eqref{eq:DS-qMZV}
as follows.
Let $\bk$ and $\bl$ be admissible indices.
Then, for $\bullet \in \{\ast, \shp\}$, we see that
\begin{align*}
  Z(z_{\bk})Z(z_{\bl})
   & =\lim_{q\to 1-0}Z_{q}(g_{\bk})Z_{q}(g_{\bl})=
  \lim_{q\to 1-0}Z_{q}(g_{\bk} \bullet g_{\bl})                                   \\
   & =Z(\iota(g_{\bk} \bullet g_{\bl}))=Z(\iota(g_{\bk}) \bullet \iota(g_{\bl}))=
  Z(z_{\bk}\bullet z_{\bl}).
\end{align*}

%%%%%%%%%%%%%%%%%%%%%%%%%%%%%%%%%%%%%%%%%%%%%%%%%

\section{A $q$-analogue of truncated SMZV}\label{sec:qSMZV-truncated}

Let $\psi^{\bullet} \, (\bullet \in \{\ast, \shp\})$ be the $\mathcal{C}$-algebra
anti-involution on
$\mathcal{C}\langle A \rangle=\mathcal{C}\langle \hbar b, ba, ba^{2}, \ldots \rangle$
defined by
\begin{align*}
   &
  \psi^{\ast}(\hbar b)=\hbar b, \quad \psi^{\ast}(ba^{k})=b(-a)^{k}, \\
   &
  \psi^{\shp}(\hbar b)=\hbar b, \quad \psi^{\shp}(ba^{k})=b(-a-\hbar)^{k}
\end{align*}
for $k \ge 1$.
For $\bullet \in \{\ast, \shp\}$, we define the $\mathcal{C}$-linear map
$w^{\mathcal{S}, \bullet}_{\hbar}:
  \mathcal{C}\langle A \rangle \to \mathcal{C}\langle A \rangle$
by $w^{\mathcal{S}, \bullet}_{\hbar}(1)=1$ and
\begin{align*}
  w^{\mathcal{S}, \bullet}_{\hbar}(u_{1} \cdots u_{r})=
  \sum_{i=0}^{r} u_{1}\cdots u_{i} \bullet_{\hbar} \psi^{\bullet}(u_{i+1} \cdots u_{r})
\end{align*}
for $r \ge 1$ and $u_{1}, \ldots , u_{r} \in A$.

Now we define the $\mathcal{C}$-linear map
$Z^{\mathcal{S}, \bullet}_{q, M}: \mathcal{C}\langle A \rangle \to \mathbb{C}$
for $\bullet \in \{\ast, \shp\}$ and $M \ge 1$ by
$Z^{\mathcal{S}, \bullet}_{q, M}=Z_{q, M}\circ w^{\mathcal{S}, \bullet}_{\hbar}$.
We call the value $Z^{\mathcal{S}, \bullet}_{q, M}(w)$ with
$w \in\mathcal{C}\langle A \rangle$
a \textit{$q$-analogue of truncated symmetric multiple zeta value}.

{}From Proposition \ref{prop:q-harmonic} and Proposition \ref{prop:q-shuffle},
we see that
\begin{align}
  Z_{q, M}^{\mathcal{S}, \ast}(u_{1}\cdots u_{r})
   & =\sum_{i=0}^{r}
  \sum_{\substack{0<m_{1}<\cdots <m_{i}<M \\
      0<m_{r}<\cdots <m_{i+1}<M}}
  \prod_{j=1}^{i}F_{q}(m_{j}; u_{j})\prod_{j=i+1}^{r}F_{q}(m_{j}; \psi^{\ast}(u_{j})),
  \label{eq:truncated-qSMZV-harmonic}     \\
  Z_{q, M}^{\mathcal{S}, \shp}(u_{1}\cdots u_{r})
   & =\sum_{i=0}^{r}
  \sum_{\substack{0<m_{1}<\cdots <m_{i}   \\ 0<m_{r}<\cdots <m_{i+1} \\ m_{i}+m_{i+1}<M}}
  \prod_{j=1}^{i}F_{q}(m_{j}; u_{j})
  \prod_{j=i+1}^{r}F_{q}(m_{j}; \psi^{\shp}(u_{j}))
  \label{eq:truncated-qSMZV-shuffle}
\end{align}
for $u_{1}, \ldots , u_{r} \in A$.

We prove the double shuffle relation of the
$q$-analogue of truncated SMZVs.
The proof is similar to that of the truncated SMZVs in \cite{OSY}.

\begin{prop}\label{prop:qS-harmonic}
  For any $w, w' \in \mathcal{C}\langle A \rangle$ and $M \ge 1$,
  it holds that
  \begin{align}
    Z^{\mathcal{S}, \ast}_{q, M}(w \ast_{\hbar} w')=Z^{\mathcal{S}, \ast}_{q, M}(w)
    Z^{\mathcal{S}, \ast}_{q, M}(w').
    \label{eq:qS-harmonic}
  \end{align}
\end{prop}

\begin{proof}
  Note that $\psi^{\ast}$ is an anti-automorphism.
  We extend the $\mathcal{C}$-linear map
  $F_{q}(m; \cdot) : \mathfrak{z} \to \mathbb{C}$
  for $m<0$ by
  $F_{q}(m; u)=F_{q}(-m; \psi^{\ast}(u))$
  for $u \in \mathfrak{z}$.
  Then, from \eqref{eq:truncated-qSMZV-harmonic}, we see that
  \begin{align}
    Z^{\mathcal{S}, \ast}_{q, M}(u_{1} \cdots u_{r})=
    \sum_{\substack{m_{1} \prec \cdots \prec m_{r} \\ 0<|m_{1}|, \ldots , |m_{r}|<M}}
    \prod_{i=1}^{r}F_{q}(m_{i}; u_{i})
    \label{eq:qSMZV-harmonic-series}
  \end{align}
  for $u_{1}, \ldots , u_{r} \in A$,
  where $\prec$ is Kontsevich's order on the set
  $(\mathbb{Z}\setminus\{0\})\sqcup\{\infty=-\infty\}$
  defined by
  $1 \prec 2 \prec \cdots \prec \infty=-\infty \prec \cdots \prec -2 \prec -1$.

    {}From the definition of $\circ_{\hbar}$ and $\psi^{\ast}$,
  we see that
  \begin{align*}
    \psi^{\ast}(u \circ_{\hbar} v)=\psi^{\ast}(u) \circ_{\hbar} \psi^{\ast}(v)
  \end{align*}
  for any $u, v \in A$.
  Hence the relation \eqref{eq:harmonic-I} holds for
  any non-zero integer $m$ and $u, v \in A$,
  and we obtain \eqref{eq:qS-harmonic}
  by using the expression \eqref{eq:qSMZV-harmonic-series}
  in the same way as the proof for the harmonic product relation of the truncated MZVs
  $\zeta_{M}(\bk)$.
\end{proof}

\begin{prop}\label{prop:shuffle-rel-q-truncated}
  For any $w, w' \in  \mathcal{C}\langle A \rangle$ and $M \ge 1$,
  it holds that
  \begin{align}
    Z^{\mathcal{S}, \shp}_{q, M}(w \,\shp_{\hbar}\, w')=
    Z^{\mathcal{S}, \shp}_{q, M}(w \, \psi^{\shp}(w')).
    \label{eq:qS-shuffle}
  \end{align}
\end{prop}

\begin{proof}
  We extend the $\mathcal{C}$-linear map
  $F_{q}(m; \cdot): \mathfrak{z} \to \mathbb{C}$
  for $m<0$ by
  $F_{q}(m; u)=F_{q}(-m; \psi^{\shp}(u))$ for $u \in A$.
  Then we see that \eqref{eq:def-Iq1}
  holds for any nonzero integer $m$ and $k \ge 1$.

  Let
  $T_{q, M}: \mathcal{C}\langle A \rangle \times \mathcal{C}\langle A \rangle \times
    \mathcal{C}\langle A \rangle \to \mathbb{C}$
  be the $\mathcal{C}$-trilinear map defined by
  $T_{q, M}(1, 1, 1)=1$ and
  \begin{align*}
     & T_{q, M}(u_{1}\cdots u_{r_{1}}, v_{1}\cdots v_{r_{2}}, w_{1}\cdots w_{r_{3}-1}) \\
     & =
    \sum_{t=1}^{3}\sum_{j=1}^{r_{t}}
    \sum_{(l_{i}^{(s)}) \in D_{j}^{(t)}}
    \prod_{i=1}^{r_{1}}F_{q}(l_{1}^{(1)}+\cdots +l_{i}^{(1)}; u_{i})
    \prod_{i=1}^{r_{2}}F_{q}(l_{1}^{(2)}+\cdots +l_{i}^{(2)}; v_{i})                   \\
     & \qquad \qquad \qquad \qquad {}\times
    \prod_{i=1}^{r_{3}-1}
    F_{q}(|l^{(1)}|+|l^{(2)}|+l_{1}^{(3)}+\cdots +l_{i}^{(3)}; w_{i}),
  \end{align*}
  where $r_{1}, r_{2} \ge 0, r_{3} \ge 1$,
  $u_{1}, \ldots , u_{r_{1}}, v_{1}, \ldots , v_{r_{2}}, w_{1}, \ldots , w_{r_{3}-1} \in A$ and
  $|l^{(s)}|=\sum_{j=1}^{r_{s}}l_{j}^{(s)} \, (s=1, 2)$.
  The summation region $D_{j}^{(t)}$ is  the subset of $\mathbb{Z}^{r_{1}+r_{2}+r_{3}}$
  consisting of tuples $(l_{i}^{(s)})_{\substack{1\le s \le 3 \\ 1 \le i \le r_{s}}}$
  satisfying the following conditions:
  \begin{align*}
    \sum_{s=1}^{3}\sum_{i=1}^{r_{s}}l_{i}^{(s)}=0, \quad -M<l_{j}^{(t)}<0, \quad
    l_{i}^{(s)}>0 \,\, \hbox{for any} \, (s, i)\not=(t, j).
  \end{align*}
  Then we see that $T_{q, M}(u, v, 1)=Z_{q, M}^{\mathcal{S}, \shp}(u \,\psi^{\shp}(v))$
  for $u, v \in \mathcal{C}\langle A \rangle$.
  Hence it suffices to show that
  \begin{align}
    T_{q, M}(u, v, w)=Z_{q, M}^{\mathcal{S}, \shp}((u \,\shp_{\hbar}\, v)w)
    \label{eq:T=Z}
  \end{align}
  for any $u, v, w \in \mathcal{C}\langle A \rangle$.

    {}From \eqref{eq:truncated-qSMZV-shuffle} and
  the definition of $T_{q, M}$,
  we see that \eqref{eq:T=Z} holds if $u=1$ or $v=1$.
  Set $u=u_{1}\cdots u_{r_{1}}$ and $v=v_{1}\cdots v_{r_{2}}$.
  We prove \eqref{eq:T=Z} by induction on $r_{1}+r_{2}$.
  Since $T_{q, M}(u, v, w)$ is symmetric with respect to $u$ and $v$,
  it suffices to consider the two cases:
  (i) $u_{r_{1}}=e_{1}-g_{1}$, \, (ii)
  $u_{r_{1}}=g_{k}, \, v_{r_{2}}=g_{l}$ with $k, l \ge 1$.
  The case (i) follows from \eqref{eq:e1g1-central} and
  $F_{q}(m; e_{1}-g_{1})=F_{q}(m+n; e_{1}-g_{1})=1-q$ for integers $m$ and $n$
  satisfying $m\not=0$ and $m+n\not=0$.
  To show the case (ii), we set
  \begin{align*}
    H_{q}(m; X)=\sum_{k\ge 1}X^{k-1}F_{q}(m; g_{k})=
    \frac{q^{m}}{[m]-q^{m}X}.
  \end{align*}
  It holds that
  \begin{align*}
     &
    H_{q}(m; X)H_{q}(n; Y)                                               \\
     & =\left\{ (1+(1-q)X)H_{q}(m; X)+(1+(1-q)Y)H_{q}(n; Y)+1-q \right\} \\
     & \times H_{q}(m+n; X+Y+(1-q)XY)
  \end{align*}
  for non-zero integers $m$ and $n$
  satisfying $m+n\not=0$.
    {}From the above relation and Lemma \ref{lem:shuffle-prod},
  we see that \eqref{eq:T=Z} holds in the case (ii) under the induction hypothesis.
\end{proof}

%%%%%%%%%%%%%%%%%%%%%%%%%%%%%%%%%%%%%%%%%%%%%%%%%

\section{A $q$-analogue of SMZV}\label{sec:qSMZV}

In this section, we define a $q$-analogue of the SMZV.
To this aim, we use the map
$Z_{q}^{\mathcal{S}, \bullet} \, (\bullet \in \{\ast, \shp\})$,
which corresponds to $\zeta^{\mathcal{S}, \bullet}$,
defined as follows.

First we define the map $Z_{q}^{\mathcal{S}, \ast}$.

\begin{prop}\label{prop:wg-harmonic}
  For any index $\bk$, the element $w_{\hbar}^{\mathcal{S}, \ast}(g_{\bk})$ belongs to
  the $\mathbb{Z}$-submodule $\oplus_{\bl \in I_{0}}\mathbb{Z}\,g_{\bl}$ of
  $\mathfrak{H}^{0}$.
\end{prop}

\begin{proof}
  {}From the definition of the harmonic product $\ast_{\hbar}$,
  we see that $g_{\bk}\ast_{\hbar}g_{\bl}=\sum_{\bm}d_{\bk,\bl}^{\,\ast, \bm}g_{\bm}$,
  where $d_{\bk,\bl}^{\ast, \bm}$ is given by \eqref{eq:def-d}.
  Hence the proof of Proposition 5.4 in \cite{OSY} with $\bullet=\ast$ and $n=0$ works
  also for our map $w_{\hbar}^{\mathcal{S}, \ast}$.
\end{proof}

\begin{defn}\label{def:qSMZV-harmonic}
  Set
  \begin{align*}
    \mathfrak{g}=\bigoplus_{\bk \in I}\mathcal{C}g_{\bk},
  \end{align*}
  which is a $\mathcal{C}$-submodule of $\widehat{\mathfrak{H}^{0}}$.
  We define the $\mathcal{C}$-linear map
  $Z_{q}^{\mathcal{S}, \ast}: \mathfrak{g} \rightarrow \mathcal{Z}_{q}$ by
  \begin{align*}
    Z_{q}^{\mathcal{S}, \ast}(g_{\bk})=\lim_{M \to \infty}Z_{q, M}^{\mathcal{S}, \ast}(g_{\bk})=
    Z_{q}(w_{\hbar}^{\mathcal{S}, \ast}(g_{\bk}))
  \end{align*}
  for an index $\bk$.
  More explicitly, we have
  \begin{align*}
    Z_{q}^{\mathcal{S}, \ast}(g_{\bk})=\sum_{i=0}^{r}(-1)^{k_{i+1}+\cdots +k_{r}}
    \sum_{\substack{0<m_{1}<\cdots <m_{i} \\ 0<m_{r}<\cdots <m_{i+1}}}
    \frac{q^{k_{1}m_{1}+\cdots +k_{r}m_{r}}}{[m_{1}]^{k_{1}} \cdots [m_{r}]^{k_{r}}}
  \end{align*}
  for a non-empty index $\bk=(k_{1}, \ldots , k_{r})$.
\end{defn}

We want to define the map $Z_{q}^{\mathcal{S}, \shp}$ similarly
by taking the limit of $Z_{q, M}^{\mathcal{S}, \shp}$ as $M \to \infty$.
For that purpose, however, we should determine the domain carefully
because it is not closed under the concatenation product
unlike the $\mathcal{C}$-module
$\mathfrak{g}$ in Definition \ref{def:qSMZV-harmonic}
as seen by the following example.

\begin{example}
  We have
  $w_{\hbar}^{\mathcal{S}, \shp}(e_{1})=-w_{\hbar}^{\mathcal{S}, \shp}(g_{1})=e_{1}-g_{1}$
  and $Z_{q, M}^{\mathcal{S}, \shp}(e_{1})=-Z_{q, M}^{\mathcal{S}, \shp}(g_{1})=
    (1-q)(M-1)$.
  Hence a $\mathcal{C}$-linear combination of $e_{1}$ and $g_{1}$
  whose image by $Z_{q, M}^{\mathcal{S}, \shp}$ converges in the limit as $M\to \infty$
  should be proportional to $e_{1}+g_{1}$.
  However, we see that
  \begin{align*}
    Z_{q, M}^{\mathcal{S}, \shp}((e_{1}+g_{1})^{2})=
    Z_{M}((e_{1}-g_{1})^{2})=(1-q)^{2}\binom{M-1}{2} \to \infty \qquad (M \to \infty).
  \end{align*}
  Therefore, the concatenation product $(e_{1}+g_{1})^{2}$ of $e_{1}+g_{1}$
  does not belong to the domain of the map
  $\lim_{M \to \infty}Z_{q, M}^{\mathcal{S}, \shp}$.
  Here we note that,
  because $w_{\hbar}^{\mathcal{S}, \shp}(e_{1}g_{1}+g_{1}e_{1})=-(e_{1}-g_{1})^{2}$,
  it holds that $Z_{q, M}^{\mathcal{S}, \shp}(e_{1}^{2}+2e_{1}g_{1}+2g_{1}e_{1}+g_{1}^{2})=0$,
  which clearly converges as $M \to \infty$.
\end{example}

To describe the domain of the map $w_{\hbar}^{\mathcal{S}, \shp}$,
we introduce the element $E_{1^{m}} \, (m\ge 0)$ of $\mathcal{C}\langle A \rangle$
defined by
\begin{align*}
  E_{1^{m}}=\frac{1}{(m+1)!}\sum_{j=0}^{m}
  g_{1}^{\shp_{\hbar} j} \shp_{\hbar} e_{1}^{\shp_{\hbar}(m-j)},
\end{align*}
where $u^{\shp_{\hbar}0}=1$ and
$u^{\shp_{\hbar}j}=u \shp_{\hbar} u^{\shp_{\hbar}(j-1)}$ for
$u \in \mathfrak{H}$ and $j \ge 1$.
For example, we have
\begin{align*}
  E_{1^{0}}=1, \qquad
  E_{1^{1}}=\frac{1}{2}(e_{1}+g_{1}), \qquad
  E_{1^{2}}=\frac{1}{6}(e_{1}^{2}+2e_{1}g_{1}+2g_{1}e_{1}+g_{1}^{2}).
\end{align*}

Let $\bk$ be a non-empty index.
It can be written uniquely in the form
\begin{align}
  \bk=(\underbrace{1, \ldots , 1}_{s_{0}}, t_{1}+2,
  \underbrace{1, \ldots , 1}_{s_{1}}, \ldots , t_{r}+2, \underbrace{1, \ldots , 1}_{s_{r}})
  \label{eq:index-form-E}
\end{align}
with $r \ge 0$ and $s_{0}, \ldots , s_{r}, t_{1}, \ldots , t_{r} \ge 0$,
where the right hand side reads $(\underbrace{1, \ldots , 1}_{s_{0}})$ if $r=0$.
Then we set
\begin{align*}
  E_{\bk}=E_{1^{s_{0}}}e_{t_{1}+2}E_{1^{s_{1}}} \cdots e_{t_{r}+2}
  E_{1^{s_{r}}}.
\end{align*}
For the empty index, we set $E_{\varnothing}=1$.

\begin{prop}\label{prop:wE-shuffle}
  For any index $\bk$,  the element $w_{\hbar}^{\mathcal{S}, \shp}(E_{\bk})$ belongs to
  $\mathfrak{H}^{0}$.
\end{prop}

See Appendix \ref{subsec:gen-E} and Appendix \ref{subsec:wE-shuffle}
for the proof.

\begin{defn}
  Set
  \begin{align*}
    \mathfrak{e}=\bigoplus_{\bk \in I}\mathcal{C}E_{\bk}.
  \end{align*}
  We define the $\mathcal{C}$-linear map
  $Z_{q}^{\mathcal{S}, \shp}: \mathfrak{e} \rightarrow \mathcal{Z}_{q}$ by
  \begin{align*}
    Z_{q}^{\mathcal{S}, \shp}(E_{\bk})=\lim_{M \to \infty}Z_{q, M}^{\mathcal{S}, \shp}(E_{\bk})=
    Z_{q}(w_{\hbar}^{\mathcal{S}, \shp}(E_{\bk})).
  \end{align*}
  for an index $\bk$.
\end{defn}

\begin{example}
  If $k\ge 2$, it holds that
  $F_{q}(m; \psi^{\shp}(e_{k}))=F_{q}(-m; e_{k})=(-1)^{k}q^{m}/[m]^{k}$
  for $m \ge 1$.
  Hence, for a non-empty index $\bk=(k_{1}, \ldots, k_{r})$
  whose all components are larger than one, we have
  \begin{align*}
    Z_{q}^{\mathcal{S}, \shp}(E_{\bk})
     & =Z_{q}^{\mathcal{S}, \shp}(e_{k_{1}} \cdots e_{k_{r}}) \\
     & =\sum_{i=0}^{r}
    (-1)^{k_{i+1}+\cdots +k_{r}}
    \sum_{\substack{0<m_{1}<\cdots <m_{i}                     \\ 0<m_{r}<\cdots <m_{i+1}}}
    \frac{q^{(k_{1}-1)m_{1}+\cdots +(k_{i}-1)m_{i}}}{[m_{1}]^{k_{1}} \cdots [m_{i}]^{k_{i}}}
    \frac{q^{m_{i+1}+\cdots +m_{r}}}{[m_{i+1}]^{k_{i+1}} \cdots [m_{r}]^{k_{r}}}.
  \end{align*}
\end{example}

\begin{example}
  {}From Proposition \ref{prop:1-RG-exp} and Lemma \ref{lem:K-shuffle1},
  we see that $w_{\hbar}^{\mathcal{S}, \shp}(E_{1^{m}})=0$ for $m \ge 1$.
  Hence we have
  $Z_{q}^{\mathcal{S}, \shp}(E_{1, \ldots , 1})=Z_{q}^{\mathcal{S}, \shp}(E_{1^{m}})=0$.
\end{example}

Using the two maps $Z_{q}^{\mathcal{S}, \ast}$ and
$Z_{q}^{\mathcal{S}, \shp}$, we define a $q$-analogue of the SMZV.
Recall that we need the relation \eqref{eq:comparison-zeta} to define the SMZV.
We show the corresponding relation in the $q$-analogue case.

Set
\begin{align*}
  \mathcal{N}_{q}=Z_{q}(\mathfrak{n}).
\end{align*}
{}From Proposition \ref{prop:n-ideal},
we see that $\mathcal{N}_{q}$ is an ideal of $\widehat{\mathcal{Z}_{q}}$.
Note that $(1-q)\mathcal{Z}_{q} \subset \mathcal{N}_{q}$
since $\hbar \, \mathfrak{H}^{0} \subset \mathfrak{n}$.

For $\bullet \in \{\ast, \shp\}$ and
$f(X) \in \mathcal{C}\langle A \rangle[[X]]$ satisfying $f(0)=0$,
we define the exponential $\exp_{\bullet_{\hbar}}(f(X))$
with respect to the product $\bullet_{\hbar}$ by
\begin{align*}
  \exp_{\bullet_{\hbar}}{\left( f(X)\right)}=1+\sum_{n \ge 1}\frac{1}{n!}
  \underbrace{f(X) \bullet_{\hbar} \cdots \bullet_{\hbar} f(X)}_{\hbox{\tiny $n$ times}}.
\end{align*}

\begin{thm}\label{thm:comparison}
  For any admissible index $\bk$, it holds that
  \begin{align}
    Z_{q}\left(g_{\bk}\,\frac{1}{1-g_{1}X}\right) \equiv
    \exp{\left( \sum_{n\ge 2} \frac{(-1)^{n-1}}{n}Z_{q}(g_{n})X^{n}\right)}
    Z_{q}\left(E_{\bk} \exp_{\shp_{\hbar}}(g_{1}X) \right)
    \label{eq:comparison}
  \end{align}
  in $\widehat{\mathcal{Z}_{q}}[[X]]$ modulo $\mathcal{N}_{q}[[X]]$.
\end{thm}

\begin{proof}
  We start from the following equality.
  See Appendix \ref{sec:comparison} for the proof.

  \begin{lem}\label{lem:comparison}
    Set
    \begin{align}
      R(X)=\frac{e^{\hbar b X}-1}{\hbar b}=\sum_{n=1}^{\infty}\frac{X^{n}}{n!}(e_{1}-g_{1})^{n-1}.
      \label{eq:def-R(X)}
    \end{align}
    For any admissible index $\bk$, it holds that
    \begin{align}
      \frac{1}{1-R(X)g_{1}}\,\shp_{\hbar}\,g_{\bk}\frac{1}{1-g_{1}X}\equiv
      \frac{1}{1-g_{1}X}\,\shp_{\hbar}\,E_{\bk}\frac{1}{1-R(X)g_{1}}
      \label{eq:comparison-alg}
    \end{align}
    modulo the $\mathcal{C}$-submodule $\mathfrak{n}[[X]]$ of $\mathfrak{H}[[X]]$.
  \end{lem}

  We calculate the image by the map $Z_{q}$
  of the both hand sides of \eqref{eq:comparison-alg}
  using Proposition \ref{prop:q-shuffle} and \eqref{eq:1-RG-exp}.
  Then we obtain
  \begin{align*}
    \exp{(Z_{q}(g_{1})X)}Z_{q}\left(g_{\bk}\,\frac{1}{1-g_{1}X}\right) \equiv
    Z_{q}\left(\frac{1}{1-g_{1}X}\right)
    Z_{q}\left(E_{\bk} \exp_{\shp_{\hbar}}(g_{1}X) \right)
  \end{align*}
  modulo $\mathcal{N}_{q}[[X]]$.
  It holds that
  \begin{align}
    \frac{1}{1-g_{k}X}=
    \exp_{\ast_{\hbar}}{\left( \sum_{n\ge 1}\frac{(-1)^{n-1}}{n}g_{nk}X^{n}\right)}
    \label{eq:1-g1X}
  \end{align}
  for any $k \ge 1$
  (see \cite[Corollary 1]{IKZ}).
  Hence, from Proposition \ref{prop:q-harmonic}, we have
  \begin{align*}
    Z_{q}\left(\frac{1}{1-g_{1}X}\right)=
    \exp{\left( \sum_{n\ge 1}\frac{(-1)^{n-1}}{n}Z_{q}(g_{n})X^{n}\right)}.
  \end{align*}
  It implies \eqref{eq:comparison} since
  $\mathcal{N}_{q}$ is an ideal of $\widehat{\mathcal{Z}_{q}}$.
\end{proof}

We denote by $\mathcal{P}_{q}$ the ideal of $\mathcal{Z}_{q}$
generated by the set $\{ Z_{q}(g_{2k}) \mid k\ge 1\}$.

\begin{cor}
  For any index $\bk$, the difference
  $Z_{q}^{\mathcal{S}, \ast}(g_{\bk})-Z_{q}^{\mathcal{S}, \shp}(E_{\bk})$
  belongs to the ideal $\mathcal{N}_{q}+\mathcal{P}_{q}$.
\end{cor}

\begin{proof}
  {}From the definition of $w_{\hbar}^{\mathcal{S}, \bullet} \, (\bullet \in \{\ast, \shp\})$,
  Proposition \ref{prop:1-RG-exp}, Corollary \ref{cor:psi(E)} and
  Lemma \ref{lem:K-shuffle1},
  we see that the difference
  $Z_{q}^{\mathcal{S}, \ast}(g_{\bk})-Z_{q}^{\mathcal{S}, \shp}(E_{\bk})$
  is a signed sum of
  \begin{align}
    \sum_{s=0}^{m}(-1)^{m-s}\left(
    Z_{q}(g_{\bm}g_{1}^{s})Z_{q}(g_{\bm'}g_{1}^{m-s})-
    Z_{q}(E_{\bm}\frac{g_{1}^{\shp_{\hbar} s}}{s!})
    Z_{q}(E_{\bm'}\frac{g_{1}^{\shp_{\hbar} (m-s)}}{(m-s)!})
    \right)
    \label{eq:qSMZV-difference}
  \end{align}
  with admissible indices $\bm, \bm'$ and $m \ge 0$ modulo
  $Z_{q}(\hbar \widehat{\mathfrak{H}^{0}})$.
  We denote \eqref{eq:qSMZV-difference} by $J_{m}$
  and calculate the generating function $J(X)=\sum_{m \ge 0}J_{m}X^{m}$.
    {}From Theorem \ref{thm:comparison}, we see that
  \begin{align*}
    J(X) & =Z_{q}(g_{\bm}\frac{1}{1-g_{1}X})Z_{q}(g_{\bm'}\frac{1}{1+g_{1}X})-
    Z_{q}(E_{\bm} \exp_{\shp_{\hbar}}{(g_{1}X)})
    Z_{q}(E_{\bm'} \exp_{\shp_{\hbar}}{(-g_{1}X)})                                                \\
         & \equiv \left\{\exp{\left(\sum_{k\ge 1} \frac{Z_{q}(g_{2k})}{k}X^{2k}\right)}-1\right\}
    Z_{q}(g_{\bm}\frac{1}{1-g_{1}X})Z_{q}(g_{\bm'}\frac{1}{1+g_{1}X})
  \end{align*}
  modulo $\mathcal{N}_{q}[[X]]$.
  The right hand side belongs to $\mathcal{P}_{q}[[X]]$.
\end{proof}

Now we are in a position to define a $q$-analogue of the SMZV.

\begin{defn}
  For an index $\bk$, we define a \textit{$q$-analogue of SMZV} ($q$SMZV)
  $\zeta_{q}^{\mathcal{S}}(\bk)$ as an element of
  the quotient $\mathcal{Z}_{q}/(\mathcal{N}_{q}+\mathcal{P}_{q})$ by
  \begin{align*}
    \zeta_{q}^{\mathcal{S}}(\bk)=Z_{q}^{\mathcal{S}, \ast}(g_{\bk})=
    Z_{q}^{\mathcal{S}, \shp}(E_{\bk})
  \end{align*}
  modulo $\mathcal{N}_{q}+\mathcal{P}_{q}$.
\end{defn}

\begin{example}
  Let $k$ be a positive integer.
    {}From the definition of $w_{\hbar}^{\mathcal{S}, \ast}$ and \eqref{eq:1-g1X},
  we see that
  \begin{align*}
    \sum_{s\ge 0}X^{s}w_{\hbar}^{\mathcal{S}, \ast}(g_{k}^{s})=
    \frac{1}{1-g_{k}X}\ast_{\hbar}\frac{1}{1-(-1)^{k}g_{k}X}=
    \exp_{\ast_{\hbar}}{\left(\sum_{n \ge 1}
    \frac{(-1)^{n-1}}{n}(1+(-1)^{kn})g_{kn}X^{n}\right)}.
  \end{align*}
  Hence Proposition \ref{prop:q-harmonic} implies that
  \begin{align*}
    \sum_{s\ge 0}X^{s}Z_{q}^{\mathcal{S}, \ast}(g_{k}^{s})=
    \exp{\left(\sum_{n \ge 1}
    \frac{(-1)^{n-1}}{n}(1+(-1)^{kn})Z_{q}(g_{kn})X^{n}\right)}.
  \end{align*}
  If $k$ is odd, then $1+(-1)^{kn}=0$ unless $n$ is even.
  Therefore, the right hand side belongs to $1+\mathcal{P}_{q}[[X]]$
  and we have
  \begin{align}
    \zeta_{q}^{\mathcal{S}}(\underbrace{k, \ldots , k}_{r})=0
    \label{eq:111-zero}
  \end{align}
  for any $k, r \ge 1$.
  In particular, the $q$SMZV of depth one is always equal to zero.
\end{example}

\begin{example}
  We consider the $q$SMZV of depth two.
  Set $\bk=(k_{1}, k_{2})$.
  If $k_{1}$ and $k_{2}$ are even, then
  $w_{\hbar}^{\mathcal{S}, \ast}(g_{\bk})=
    2g_{k_{1}}\ast_{\hbar}g_{k_{2}}-g_{k_{1}+k_{2}}$.
  If $k_{1}$ and $k_{2}$ are odd, we have
  $w_{\hbar}^{\mathcal{S}, \ast}(g_{\bk})=-g_{k_{1}+k_{2}}$.
  In both cases we see that $Z_{q}(w_{\hbar}^{\mathcal{S}, \ast}(g_{\bk}))$
  belongs to $\mathcal{P}_{q}$ from Proposition \ref{prop:q-harmonic}.
  Therefore $\zeta_{q}^{\mathcal{S}}(k_{1}, k_{2})=0$ if $k_{1}+k_{2}$ is even.

  We consider the case where $k_{1}+k_{2}$ is odd.
  To this aim we calculate the $q$MZV $Z_{q}(g_{k_{1}}g_{k_{2}})$
  whose weight is odd modulo $\mathcal{N}_{q}+\mathcal{P}_{q}$.
  The calculation is similar to that for MZV in \cite{Cartier, Zagier23}.

  Suppose that $k$ is odd and $k \ge 3$.
  For $1\le m<k$, we have
  \begin{align*}
    g_{m}\ast_{\hbar}g_{k-m}=g_{m}g_{k-m}+g_{k-m}g_{m}+g_{k}.
  \end{align*}
  {}From Lemma \ref{lem:shuffle-prod}, we also see that
  \begin{align*}
    \sum_{m, l \ge 1}X^{m-1}Y^{l-1}g_{m}\,\shp_{\hbar}\,g_{l}
     & =
    (1+\hbar X)\sum_{m, l\ge 1}X^{k-1}(X+Y+\hbar XY)^{l-1}g_{m}g_{l}     \\
     & +(1+\hbar Y)\sum_{m, l\ge 1}Y^{k-1}(X+Y+\hbar XY)^{l-1}g_{m}g_{l} \\
     & +\sum_{l \ge 1}(X+Y+\hbar XY)^{l-1}(e_{1}-g_{1})g_{l}.
  \end{align*}
  Since $k \ge 3$ and $(e_{1}-g_{1})g_{l}$ belongs to $\mathfrak{n}_{0}$ if $l \ge 2$,
  we have
  \begin{align*}
    g_{m}\,\shp_{\hbar}\,g_{k-m}\equiv
    \sum_{j \ge 1}\left(\binom{k-j-1}{k-m-1}+\binom{k-j-1}{m-1}\right)g_{j}g_{k-j}
  \end{align*}
  modulo $\mathfrak{n}$.
  Note that $m$ or $k-m$ is even.
  Hence, from Proposition \ref{prop:q-harmonic} and Proposition \ref{prop:q-shuffle},
  we see that
  \begin{align}
     & Z_{q}(g_{m}g_{k-m}+g_{k-m}g_{m}+g_{k})\equiv 0,
    \label{eq:depth-two-1}                                                                   \\
     & \sum_{j \ge 1}\left(\binom{k-j-1}{k-m-1}+\binom{k-j-1}{m-1}\right)Z_{q}(g_{j}g_{k-j})
    \equiv 0
    \label{eq:depth-two-2}
  \end{align}
  modulo $\mathcal{N}_{q}+\mathcal{P}_{q}$.
  Set
  \begin{align*}
    \mathcal{D}(X, Y)=\sum_{m=1}^{k-1}X^{m-1}Y^{k-m-1}Z_{q}(g_{m}g_{k-m}), \qquad
    \mathcal{Z}(X, Y)=\frac{X^{k-1}-Y^{k-1}}{X-Y}Z_{q}(g_{k}).
  \end{align*}
  Then \eqref{eq:depth-two-1} and \eqref{eq:depth-two-2} imply that
  \begin{align*}
    \mathcal{D}(X, Y)+\mathcal{D}(Y, X)+\mathcal{Z}(X, Y)\equiv 0, \qquad
    \mathcal{D}(X, X+Y)+\mathcal{D}(Y, X+Y)\equiv 0
  \end{align*}
  modulo $(\mathcal{N}_{q}+\mathcal{P}_{q})[X, Y]$, respectively.
  Using these relations and $\mathcal{D}(-X, -Y)=-\mathcal{D}(X, Y)$, we obtain
  \begin{align*}
    \mathcal{D}(X, Y)\equiv -\frac{1}{2}\left(
    \mathcal{Z}(X, Y)+\mathcal{Z}(X-Y, X)+\mathcal{Z}(-Y, X-Y) \right).
  \end{align*}
  Hence we find that
  \begin{align}
    Z_{q}(g_{m}g_{k-m})\equiv -\frac{1}{2}\left(1+(-1)^{m}\binom{k}{m}\right)Z_{q}(g_{k})
    \label{eq:depth-two-3}
  \end{align}
  modulo $\mathcal{N}_{q}+\mathcal{P}_{q}$ if $1\le m<k$ and $k$ is odd.

  Now suppose that $k_{1}, k_{2}\ge 1$ and $k_{1}+k_{2}$ is odd.
  Since $k_{1}$ or $k_{2}$ is even, we see that
  \begin{align*}
    Z_{q}^{\mathcal{S}, \ast}(g_{k_{1}}g_{k_{2}})
     & =  Z_{q}(g_{k_{1}}g_{k_{2}}+(-1)^{k_{2}}g_{k_{1}}\ast_{\hbar}g_{k_{2}}-g_{k_{2}}g_{k_{1}}) \\
     & \equiv Z_{q}(g_{k_{1}}g_{k_{2}}-g_{k_{2}}g_{k_{1}})                                        \\
     & \equiv (-1)^{k_{2}}\binom{k_{1}+k_{2}}{k_{1}}Z_{q}(g_{k_{1}+k_{2}})
  \end{align*}
  modulo $\mathcal{N}_{q}+\mathcal{P}_{q}$ by using Proposition \ref{prop:q-harmonic}
  and \eqref{eq:depth-two-3}.

  From the above arguments we see that
  \begin{align*}
    \zeta_{q}^{\mathcal{S}}(k_{1}, k_{2})=(-1)^{k_{2}}\binom{k_{1}+k_{2}}{k_{1}}
    Z_{q}(g_{k_{1}+k_{2}})
  \end{align*}
  in the quotient $\mathcal{Z}_{q}/(\mathcal{N}_{q}+\mathcal{P}_{q})$
  for any $k_{1}, k_{2} \ge 1$.
  It is a $q$-analogue of the formula for the SMZV of depth two
  \begin{align*}
    \zeta^{\mathcal{S}}(k_{1}, k_{2})=(-1)^{k_{2}}\binom{k_{1}+k_{2}}{k_{1}}
    \zeta(k_{1}+k_{2})
  \end{align*}
  modulo $\zeta(2)\mathcal{Z}$ (see, e.g., \cite{Kaneko}).
\end{example}

We check that our $q$SMZV is really a $q$-analogue of the SMZV.

\begin{thm}\label{thm:limit-qSMZV}
  For any index $\bk$, it holds that
  \begin{align*}
    \lim_{q \to 1-0}Z_{q}^{\mathcal{S}, \ast}(g_{\bk})=\zeta^{\mathcal{S}, \ast}(\bk), \qquad
    \lim_{q \to 1-0}Z_{q}^{\mathcal{S}, \shp}(E_{\bk})=\zeta^{\mathcal{S}, \shp}(\bk).
  \end{align*}
\end{thm}

\begin{proof}
  We see that $\iota(\psi^{\bullet}(w))=\psi(\iota(w))$ for any
  $w \in \mathcal{C}\langle A \rangle$ from the definition of
  $\iota, \psi^{\bullet}$ and $\psi$.
  Hence, Proposition \ref{prop:iota-prod} implies that
  $\iota(w_{\hbar}^{\mathcal{S}, \bullet}(w))=w^{\mathcal{S}, \bullet}(\iota(w))$
  for any $w \in \mathcal{C}\langle A \rangle$ and $\bullet \in \{\ast, \shp\}$.
    {}From the definition of $\iota$, we have $\iota(g_{\bk})=z_{\bk}$
  for any index $\bk$.
  Moreover, we have $\iota(e_{k})=z_{k}$ for $k \ge 1$ and
  \begin{align*}
    \iota(E_{1^{m}})=\frac{1}{(m+1)!}\sum_{j=0}^{m}
    \iota(g_{1})^{\shp j} \,\shp\,\iota(e_{1})^{\shp (m-j)}=
    \frac{z_{1}^{\shp m}}{m!}=z_{1}^{m}
  \end{align*}
  for $m \ge 0$. Hence $\iota(E_{\bk})=z_{\bk}$ for any index $\bk$.
  Now the desired formula follows from Corollary \ref{cor:qMZV-lim}.
\end{proof}

The limit as $q \to 1-0$ of
any element of $\mathcal{N}_{q}$ is zero
and that of $\mathcal{P}_{q}$ is contained in
$\zeta(2)\mathcal{Z}$ because
$\lim_{q\to 1-0}Z_{q}(g_{2k})=\zeta(2k) \in \mathbb{Q}\, \zeta(2)^{k}$ for $k \ge 1$.
Therefore, we have the well-defined map
$\mathcal{Z}_{q}/(\mathcal{N}_{q}+\mathcal{P}_{q}) \to
  \mathcal{Z}/\zeta(2)\mathcal{Z}$
which sends the equivalent class of $f(q) \in \mathcal{Z}_{q}$ to
that of $\lim_{q\to 1-0}f(q)$.
Theorem \ref{thm:limit-qSMZV} implies that
the map sends the $q$SMZV $\zeta_{q}^{\mathcal{S}}(\bk)$ to
the SMZV $\zeta^{\mathcal{S}}(\bk)$.
In this sense we may regard $\zeta_{q}^{\mathcal{S}}(\bk)$ as
a $q$-analogue of $\zeta^{\mathcal{S}}(\bk)$.

%%%%%%%%%%%%%%%%%%%%%%%%%%%%%%%%%%%%%%%%%%

\section{Relations of the $q$-analogue of symmetric multiple zeta value}\label{sec:relations}

\subsection{Reversal relation}

\begin{thm}
  For any index $\bk$, we have
  \begin{align*}
    Z_{q}^{\mathcal{S}, \ast}(g_{\overline{\bk}})=(-1)^{\mathrm{wt}(\bk)}
    Z_{q}^{\mathcal{S}, \ast}(g_{\bk}),
  \end{align*}
  and
  \begin{align*}
    Z_{q}^{\mathcal{S}, \shp}(E_{\overline{\bk}})\equiv (-1)^{\mathrm{wt}(\bk)}
    Z_{q}^{\mathcal{S}, \shp}(E_{\bk})
  \end{align*}
  modulo $(1-q)\mathcal{Z}_{q}$.
  Therefore, for any index $\bk$, it holds that
  \begin{align*}
    \zeta_{q}^{\mathcal{S}}(\overline{\bk})=(-1)^{\mathrm{wt}(\bk)}
    \zeta_{q}^{\mathcal{S}}(\bk).
  \end{align*}
\end{thm}

\begin{proof}
  Since $\psi^{\bullet}$ is an anti-involution,
  we see that $w_{\hbar}^{\mathcal{S}, \bullet}(\psi^{\bullet}(w))=
    w_{\hbar}^{\mathcal{S}, \bullet}(w)$
  for any $w \in \mathcal{C}\langle A \rangle$ and $\bullet \in \{\ast, \shp\}$.
    {}From the definition of $\psi^{\ast}$, we have
  $\psi^{\ast}(g_{\bk})=(-1)^{\mathrm{wt}(\bk)}g_{\overline{\bk}}$.
  We also have $\psi^{\shp}(E_{\bk})\equiv(-1)^{\mathrm{wt}(\bk)}E_{\overline{\bk}}$
  modulo $\hbar \, \mathfrak{e}$ (see Corollary \ref{cor:psi(E)}).
  Thus we obtain the desired equalities.
\end{proof}

\subsection{Double shuffle relation}

{}From Proposition \ref{prop:qS-harmonic}, we obtain
the following relation of $q$SMZVs.

\begin{prop}\label{prop:harmonic-rel-qSMZV}
  For any index $\bk$ and $\bl$, it holds that
  \begin{align*}
    Z_{q}^{\mathcal{S}, \ast}(g_{\bk} \ast_{\hbar} g_{\bl})=
    Z_{q}^{\mathcal{S}, \ast}(g_{\bk})
    Z_{q}^{\mathcal{S}, \ast}(g_{\bl}).
  \end{align*}
\end{prop}

Next we consider the shuffle relation.
Note that the product $E_{\bk}\,\shp_{\hbar}\,E_{\bl}$ does not
necessarily belong to $\mathfrak{e}$ as follows.

\begin{example}\label{ex:shuffle-not-closed}
  We have
  \begin{align*}
    E_{1}\,\shp_{\hbar}\,E_{1}=\frac{1}{4}(g_{1}+e_{1})\,\shp_{\hbar}\,(g_{1}+e_{1})=
    \frac{1}{4}(g_{1}^{2}+3g_{1}e_{1}+3e_{1}g_{1}+e_{1}^{2}),
  \end{align*}
  which is not a $\mathcal{C}$-linear combination of
  $E_{11}=E_{1^{2}}=(e_{1}^{2}+2e_{1}g_{1}+2g_{1}e_{1}+g_{1}^{2})/6, E_{2}=e_{2}$
  and $E_{1}=(e_{1}+g_{1})/2$.
\end{example}

However, we have the following proposition.
We set
\begin{align*}
  \mathfrak{e}^{0}=\sum_{\bk \in I_{0}}\mathcal{C}E_{\bk}.
\end{align*}

\begin{prop}\label{prop:E-close}
  Let $\bk$ and $\bl$ be an index.
  If at least one of $\bk$ and $\bl$ is admissible,
  then $E_{\bk} \,\shp_{\hbar}\, E_{\bl}$ belongs to $\mathfrak{e}$.
  If both $\bk$ and $\bl$ are admissible,
  then $E_{\bk} \,\shp_{\hbar}\, E_{\bl}$ belongs to $\mathfrak{e}^{0}$.
\end{prop}

See Appendix \ref{sec:app-E-close} for the proof.

\begin{prop}\label{prop:shuffle-rel-qSMZV}
  Let $\bk$ and $\bl$ be an index and suppose that at least one of them is admissible.
  Then it holds that
  \begin{align*}
    Z_{q}^{\mathcal{S}, \shp}(E_{\bk} \shp_{\hbar} E_{\bl})\equiv
    (-1)^{\mathrm{wt}(\bl)}
    Z_{q}^{\mathcal{S}, \shp}(E_{(\bk, \overline{\bl})}),
  \end{align*}
  modulo $(1-q)\mathcal{Z}_{q}$,
  where $(\bk, \overline{\bl})$ is the concatenation of $\bk$ and $\overline{\bl}$.
\end{prop}

\begin{proof}
  Proposition \ref{prop:shuffle-rel-q-truncated} implies that
  $Z_{q}^{\mathcal{S}, \shp}(E_{\bk}\,\shp_{\hbar}\,E_{\bl})=
    Z_{q}^{\mathcal{S}, \shp}(E_{\bk}\psi^{\shp}(E_{\bl}))$.
    {}From Corollary \ref{cor:psi(E)}, we see that
  $E_{\bk}\psi^{\shp}(E_{\bl})\equiv(-1)^{\mathrm{wt}(\bl)}E_{(\bk, \overline{\bl})}$
  modulo $\hbar \, \mathfrak{e}$ if $\bk$ or $\bl$ is admissible.
\end{proof}

{}From Proposition \ref{prop:harmonic-rel-qSMZV} and
Proposition \ref{prop:shuffle-rel-qSMZV},
we obtain the double shuffle relation of $q$SMZVs as follows.

\begin{thm}\label{thm:DS-qSMZV}
  Let $d_{\bk, \bl}^{\, \bullet, \bm}$ be
  the non-negative integer defined by \eqref{eq:def-d}.
  \begin{enumerate}
    \item     For any index $\bk$ and $\bl$, it holds that
          \begin{align*}
            \zeta_{q}^{\mathcal{S}}(\bk)\zeta_{q}^{\mathcal{S}}(\bl)=
            \sum_{\bm}d_{\bk,\bl}^{\, \ast, \bm}\zeta_{q}^{\mathcal{S}}(\bm).
          \end{align*}
    \item Moreover, if at least one of $\bk$ and $\bl$ is admissible,
          it holds that
          \begin{align*}
            (-1)^{\mathrm{wt}(\bl)}
            \zeta_{q}^{\mathcal{S}}(\bk, \overline{\bl})=
            \sum_{\bm}d_{\bk,\bl}^{\, \shp, \bm}\zeta_{q}^{\mathcal{S}}(\bm).
          \end{align*}
  \end{enumerate}
\end{thm}

\begin{proof}
  {}From the definition of the harmonic product $\ast_{\hbar}$,
  we see that $g_{\bk}\ast_{\hbar}g_{\bl}=\sum_{\bm}d_{\bk,\bl}^{\, \ast, \bm}g_{\bm}$.
  Hence (i) is true because of Proposition \ref{prop:harmonic-rel-qSMZV}.

  We prove (ii) by using the map $\iota$ defined by \eqref{eq:def-iota}.
  Proposition \ref{prop:E-close} implies that there exists
  $c_{\bk, \bl}^{\bm}(\hbar) \in \mathbb{Q}[\hbar]$ such that
  $E_{\bk}\,\shp_{\hbar}\,E_{\bl}=\sum_{\bm}c_{\bk,\bl}^{\bm}(\hbar)E_{\bm}$.
  As shown in the proof of Theorem \ref{thm:limit-qSMZV},
  we have $\iota(E_{\bk})=z_{\bk}$ for any index $\bk$.
  Hence we see that
  \begin{align*}
    \iota(E_{\bk}\,\shp_{\hbar}\,E_{\bl})=\sum_{\bm}c_{\bk,\bl}^{\bm}(0)z_{\bm}.
  \end{align*}
  On the other hand, Proposition \ref{prop:iota-prod} implies that
  \begin{align*}
    \iota(E_{\bk}\,\shp_{\hbar}\,E_{\bl})=z_{\bk}\,\shp\,z_{\bl}=
    \sum_{\bm}d_{\bk,\bl}^{\,\shp,\bm}z_{\bm}.
  \end{align*}
  Hence $c_{\bk,\bl}^{\bm}(0)=d_{\bk,\bl}^{\,\shp,\bm}$ and
  \begin{align*}
    E_{\bk}\,\shp_{\hbar}\,E_{\bl}-\sum_{\bm}d_{\bk,\bl}^{\,\shp,\bm}E_{\bm} \in
    \hbar \, \mathfrak{e}.
  \end{align*}
  Therefore we see that (ii) is true from Proposition \ref{prop:shuffle-rel-qSMZV}
  and $(1-q)\mathcal{Z}_{q} \subset \mathcal{N}_{q}$.
\end{proof}

\subsection{Ohno-type relation}

{}From Theorem \ref{thm:DS-qSMZV}, we see that
the $q$SMZVs satisfy a part of the Ohno-type relation.

\begin{thm}
  Let $\bk$ be a non-empty admissible index.
  Then
  it holds that
  \begin{align}
    \sum_{\substack{\be \in (\mathbb{Z}_{\ge 0})^{r} \\ \mathrm{wt}(\be)=m}}
    \zeta_{q}^{\mathcal{S}}(\bk+\be)=
    \sum_{\substack{\be \in (\mathbb{Z}_{\ge 0})^{s} \\ \mathrm{wt}(\be)=m}}
    \zeta_{q}^{\mathcal{S}}((\bk^{\vee}+\be)^{\vee})
    \label{eq:Ohno-qMSZV}
  \end{align}
  for any $m \ge 1$, where $r=\mathrm{dep}(\bk)$ and $s=\mathrm{dep}(\bk^{\vee})$.
\end{thm}

\begin{proof}
  We define the $\mathbb{Q}$-linear map
  $\eta: \mathfrak{h}^{1} \rightarrow \mathcal{Z}_{q}/(\mathcal{N}_{q}+\mathcal{P}_{q})$ by
  $\eta(z_{\bk})=\zeta^{\mathcal{S}}_{q}(\bk)$ for an index $\bk$.
  Theorem \ref{thm:DS-qSMZV} and \eqref{eq:111-zero} imply the following properties:
  \begin{enumerate}
    \item For any index $\bk$ and $n\ge 1$, it holds that
          $\eta(z_{\bk}\ast z_{1}^{n})=0$.
    \item For any admissible index $\bk$ and $n \ge 0$, it holds that
          $\eta(z_{\bk}\,\shp\,z_{1}^{n})=(-1)^{n} \eta(z_{\bk}z_{1}^{n})$.
  \end{enumerate}

  For a non-empty index $\bk=(k_{1}, \ldots , k_{r})$ and $s\ge 0$,
  we define the element $a_{s}(\bk)$ of $\mathfrak{h}^{1}$ by
  \begin{align*}
    a_{s}(\bk)=\sum \prod_{1\le i \le r}^{\curvearrowright}
    \left( y \prod_{1\le l \le k_{i}}^{\curvearrowright}(y^{j_{l}^{(i)}} x) \right),
  \end{align*}
  where the sum is over the set
  \begin{align*}
    \left\{ (j_{l}^{(i)})_{\substack{1\le i \le r \\ 1\le l \le k_{i}}} \in
    (\mathbb{Z}_{\ge 0})^{\mathrm{wt}(\bk)} \mid
    \sum_{i=1}^{r}\sum_{l=1}^{k_{i}}j_{l}^{(i)}=s \right\}.
  \end{align*}
  We also set
  \begin{align*}
    A_{\bk, s, p}=\sum_{\substack{\lambda_{1}, \ldots , \lambda_{r} \in \{0, 1\} \\
        \lambda_{1}+\cdots +\lambda_{r}=p}}
    a_{s}(k_{1}+\lambda_{1}-1, \ldots , k_{r}+\lambda_{r}-1)
  \end{align*}
  for $s, p \ge 0$.
  Then we have
  \begin{align*}
    \sum_{p=0}^{\min(n, r)}
    \sum_{\substack{m+s=n-p \\ m, s\ge 0}}(-1)^{s}A_{\bk, s, p}\,\shp\,z_{1}^{m}=
    z_{\bk}\ast z_{1}^{n}
  \end{align*}
  for any $n \ge 1$  (see equation (3) in \cite{Oyama}).

  Now suppose that $\bk$ is admissible.
  Then $A_{\bk, s, p}$ belongs to $\sum_{\bl \in I_{0}}\mathbb{Q}z_{\bl}$
  for any $s, p \ge 0$.
  Therefore we can use the properties (i) and (ii), and see that
  \begin{align*}
    \sum_{p=0}^{\min(n, r)}(-1)^{p}
    \sum_{\substack{m+s=n-p \\ m, s\ge 0}}\eta(A_{\bk, s, p}z_{1}^{m})=0
  \end{align*}
  for any $n \ge 1$.
    {}From the above identity, we obtain \eqref{eq:Ohno-qMSZV} in the same way
  as the proof of Theorem 2.5 in \cite{Oyama}.
\end{proof}

%%%%%%%%%%%%%%%%%%%%%%%%%%%%%%%%%%%%%%%%%%
%%%%%%%%%%%%%%%%%%%%%%%%%%%%%%%%%%%%%%%%%%

\appendix

\section{Proof of Proposition \ref{prop:limit-of-qMZV}}\label{sec:app-limit}

\begin{lem}\label{lem:app-non-decreasing}
  Suppose that $\alpha\ge 0$ and $\max{(1, \alpha)}\le \beta \le 2\alpha+1$.
  Then the function $f(x)=x^{\alpha}(1-x)/(1-x^{\beta})$ is non-decreasing
  on the interval $[0, 1)$.
\end{lem}

\begin{proof}
  If $\beta=1$, the statement is trivial.
  We assume that $\beta>1$. From the assumption we see that $\alpha\ge (\beta-1)/2>0$.
  Set
  \begin{align*}
    g(x)=x^{1-\alpha}(1-x^{\beta})^{2}f'(x)=
    \alpha-(\alpha+1)x+(\beta-\alpha)x^{\beta}+(\alpha-\beta+1)x^{\beta+1}.
  \end{align*}
  Since $g(0)=\alpha>0$ and $g(1)=0$, it suffices to show that
  $g'(x)<0$ on the interval $(0, 1)$.
  Set
  \begin{align*}
    h(x)=x^{-\beta}g'(x)=-(\alpha+1)x^{-\beta}+\beta(\beta-\alpha)x^{-1}+
    (\alpha-\beta+1)(\beta+1).
  \end{align*}
  Then we see that $h(1)=0$ and, from the assumption,
  \begin{align*}
    h'(x)=\beta x^{-\beta-1}(\alpha+1-(\beta-\alpha)x^{\beta-1})>
    \beta x^{-\beta-1}(2\alpha-\beta+1)\ge 0
  \end{align*}
  for $0<x<1$.
  Therefore we see that $h(x)<0$, and hence $g'(x)<0$, on the interval $(0, 1)$.
\end{proof}

\begin{cor}\label{cor:app-non-decreasing}
  Under the assumption of Lemma \ref{lem:app-non-decreasing},
  it holds that
  \begin{align*}
    x^{\alpha}\frac{1-x}{1-x^{\beta}} \le \frac{1}{\beta}
  \end{align*}
  for $0<x<1$.
\end{cor}

\begin{proof}
  It is because $f(x) \to 1/\beta$ as $x \to 1-0$.
\end{proof}

\begin{prop}\label{prop:estimateB-1}
  Suppose that $0<q<1$.
  Set
  \begin{align*}
    B_{q}(\alpha_{1}, \ldots , \alpha_{r}; \beta_{1}, \ldots , \beta_{r})=
    \sum_{l_{1}, \ldots , l_{r}\ge 1}
    \prod_{j=1}^{r}\binom{l_{j}}{\alpha_{j}}
    \frac{(1-q)^{\alpha_{j}} q^{l_{j}/2}}{(l_{1}+\cdots +l_{j})^{\beta_{j}+1}}.
  \end{align*}
  Then, for any non-negative integer
  $\alpha_{1}, \ldots , \alpha_{r}, \beta_{1}, \ldots , \beta_{r}$, it holds that
  \begin{align*}
    0\le Z_{q}((e_{1}-g_{1})^{\alpha_{1}}g_{\beta_{1}+1} \cdots
    (e_{1}-g_{1})^{\alpha_{r}}g_{\beta_{r}+1}) \le
    B_{q}(\alpha_{1}, \ldots , \alpha_{r}; \beta_{1}, \ldots , \beta_{r}).
  \end{align*}
\end{prop}

\begin{proof}
  {}From Corollary \ref{cor:app-non-decreasing}, we see that
  \begin{align*}
    0\le \frac{q^{n}}{[n]}=q^{(n-1)/2}\frac{1-q}{1-q^{n}}q^{(n+1)/2} \le \frac{q^{(n+1)/2}}{n}
  \end{align*}
  for any $n \ge 1$.
  Therefore, it holds that
  \begin{align*}
    0 & \le Z_{q}((e_{1}-g_{1})^{\alpha_{1}}g_{\beta_{1}+1} \cdots
    (e_{1}-g_{1})^{\alpha_{r}}g_{\beta_{r}+1})                     \\
      & =\sum_{0=n_{0}<n_{1}<\cdots <n_{r}}\prod_{j=1}^{r}
    (1-q)^{\alpha_{j}}\binom{n_{j}-n_{j-1}-1}{\alpha_{j}}
    \left(\frac{q^{n_{j}}}{[n_{j}]}\right)^{\beta_{j}+1}           \\
      & \le \sum_{0=n_{0}<n_{1}<\cdots <n_{r}}\prod_{j=1}^{r}
    (1-q)^{\alpha_{j}}\binom{n_{j}-n_{j-1}}{\alpha_{j}}
    \left(\frac{q^{(n_{j}+1)/2}}{n_{j}}\right)^{\beta_{j}+1}.
  \end{align*}
  Set $l_{j}=n_{j}-n_{j-1}$ for $1\le j \le r$.
  Then the right-hand side is dominated by
  $B_{q}(\alpha_{1}, \ldots , \alpha_{r}; \beta_{1}, \ldots , \beta_{r})$
  since $0<q<1$.
\end{proof}

\begin{prop}\label{prop:estimateB-2}
  Suppose that $0<q<1$ and
  $\alpha_{1}, \ldots , \alpha_{r}, \beta_{1}, \ldots , \beta_{r}$ are
  non-negative integers.
  Then it holds that
  \begin{align}
    B_{q}(\alpha_{1}, \ldots , \alpha_{r}; \beta_{1}, \ldots , \beta_{r})=
    O((-\log{(1-q)})^{r}) \qquad (q \to 1-0).
    \label{eq:estimateB-21}
  \end{align}
  Moreover, if $\alpha_{s} \ge 1$ and $\beta_{t}\ge 1$ for some $1\le s\le t \le r$,
  then we have
  \begin{align}
    B_{q}(\alpha_{1}, \ldots , \alpha_{r}; \beta_{1}, \ldots , \beta_{r})=
    (1-q)\, O((-\log{(1-q)})^{r}) \qquad (q \to 1-0).
    \label{eq:estimateB-22}
  \end{align}
\end{prop}

\begin{proof}
  For a non-negative integer $\alpha$, we define
  \begin{align*}
    \varphi_{\alpha}(x)=(1-x)^{\alpha}\sum_{l\ge 1}\binom{l}{\alpha}\frac{x^{l/2}}{l},
    \qquad
    \tilde{\varphi}_{\alpha}(x)=(1-x)^{\alpha}
    \sum_{l\ge 1}\binom{l}{\alpha}\frac{x^{l/2}}{l^2}.
  \end{align*}
  We have
  \begin{align*}
    \varphi_{0}(x)=-\log{(1-x^{1/2})}, \qquad
    \varphi_{\alpha}(x)=\frac{x^{\alpha/2}}{\alpha}(1+x^{1/2})^{\alpha} \quad
    (\alpha \ge 1).
  \end{align*}
  Hence $\varphi_{\alpha}(x)=O(-\log{(1-x)})$ as $x \to 1-0$ for any $\alpha \ge 0$.
  If $\alpha \ge 1$, we see that
  \begin{align*}
    0\le \tilde{\varphi}_{\alpha}(x)=(1-x)^{\alpha}
    \sum_{l\ge 1}\frac{1}{\alpha}\binom{l-1}{\alpha-1}\frac{x^{l/2}}{l}
    \le \frac{1-x}{\alpha}\varphi_{\alpha-1}(x)
  \end{align*}
  for $0\le x \le 1$.
  Hence $\tilde{\varphi}_{\alpha}(x)=(1-x)\,O(-\log{(1-x)})$ as
  $x \to 1-0$ for any $\alpha \ge 1$.

  Since $\beta_{j}\ge 0$ for $1\le j \le r$, we see that
  \begin{align*}
    0\le  B_{q}(\alpha_{1}, \ldots , \alpha_{r}; \beta_{1}, \ldots , \beta_{r}) \le
    \prod_{j=1}^{r}\varphi_{\alpha_{j}}(q).
  \end{align*}
  Hence we have \eqref{eq:estimateB-21}.
  Now assume further that $\alpha_{s} \ge 1$ and $\beta_{t}\ge 1$
  for some $1\le s\le t \le r$.
  If $s=t$, it holds that
  \begin{align*}
    0 \le B_{q}(\alpha_{1}, \ldots , \alpha_{r}; \beta_{1}, \ldots , \beta_{r})
    \le \tilde{\varphi}_{\alpha_{s}}(q)
    \prod_{\substack{1\le j \le r \\ j\not=s}}\varphi_{\alpha_{j}}(q).
  \end{align*}
  Hence we obtain \eqref{eq:estimateB-22} if $s=t$.
  If $s<t$, we see that
  \begin{align*}
    0 & \le B_{q}(\alpha_{1}, \ldots , \alpha_{r}; \beta_{1}, \ldots , \beta_{r}) \\
      & \le \sum_{l_{1}, \ldots , l_{r}\ge 1}
    \prod_{\substack{1\le j \le r                                                 \\ j\not=s, t}}
    \left\{(1-q)^{\alpha_{j}}
    \binom{l_{j}}{\alpha_{j}}\frac{q^{l_{j}/2}}{l_{j}}\right\}
    (1-q)^{\alpha_{s}+\alpha_{t}}
    \binom{l_{s}}{\alpha_{s}}\binom{l_{t}}{\alpha_{t}}
    \frac{q^{l_{s}/2}}{(l_{1}+\cdots +l_{s})^{\beta_{s}+\beta_{t}+1}}
    \frac{q^{l_{t}/2}}{l_{1}+\cdots +l_{t}}                                       \\
      & \le \sum_{l_{1}, \ldots , l_{r}\ge 1}
    \prod_{\substack{1\le j \le r                                                 \\ j\not=s}}
    \left\{(1-q)^{\alpha_{j}}
    \binom{l_{j}}{\alpha_{j}}\frac{q^{l_{j}/2}}{l_{j}}\right\}
    (1-q)^{\alpha_{s}}\binom{l_{s}}{\alpha_{s}}\frac{q^{l_{s}/2}}{l_{s}^{2}}      \\
      & =\tilde{\varphi}_{\alpha_{s}}(q)
    \prod_{\substack{1\le j \le r                                                 \\ j\not=s}}
    \varphi_{\alpha_{j}}(q)
  \end{align*}
  since $\beta_{t} \ge 1$.
  Thus we get \eqref{eq:estimateB-22} in the case where $s<t$ either.
\end{proof}

\begin{proof}[Proof of Proposition \ref{prop:limit-of-qMZV}]
  {}From Lemma \ref{lem:app-non-decreasing} and the monotone convergence theorem,
  we see that $Z_{q}(g_{\bk}) \to \zeta(\bk)$ as $q \to 1-0$ for any admissible index $\bk$.
  Suppose that $w$ is an element of $\widehat{\mathfrak{H}^{0}}$
  of the form \eqref{eq:basis-monomial}.
  Proposition \ref{prop:estimateB-1} and \eqref{eq:estimateB-21} imply that
  $Z_{q}(w)=O((-\log{(1-q)})^{r})$ as $q \to 1-0$.
  Hence $Z_{q}(\hbar w)\to 0$ in the limit as $q \to 1-0$.
  If $w \in \mathfrak{n}_{0}$, then there exist $s$ and $t$ such that
  $\alpha_{s}\ge 1, \beta_{t}\ge 1$ and $1\le s\le t \le r$.
  Then Proposition \ref{prop:estimateB-1} and \eqref{eq:estimateB-22} imply that
  $Z_{q}(w)\to 0$ in the limit as $q \to 1-0$.
  Thus we see that $\lim_{q\to 1-0}Z_{q}(w)=0$ for any $w \in \mathfrak{n}$.
\end{proof}

%%%%%%%%%%%%%%%%%%%%%%%%%%%%%%%%%%%%%%%%%%

\section{Proofs}\label{sec:app-proofs}

\subsection{Formulas of shuffle product}\label{subsec:app-shpUV}

\begin{prop}\label{prop:U1}
  Suppose that $U(X)$ and $V(X)$ belong to the ideal $X \mathcal{C}[b][[X]]$
  of the formal power series ring $\mathcal{C}[b][[X]]$ generated by $X$.
  Then it holds that
  \begin{align*}
     &
    \left(
    w\frac{1}{1-U(X)a} \, \shp_{\hbar} \, w'\frac{1}{1-V(Y)a}
    \right)
    \left\{1-(U(X)+V(Y)+\hbar \, U(X)V(Y))a \right\}                                     \\
     & =-w\,\shp_{\hbar}\,w'+\left(w\frac{1}{1-U(X)a}\,\shp_{\hbar}\,w'\right)(1-U(X)a)+
    \left(w\,\shp_{\hbar}\,w'\frac{1}{1-V(Y)a}\right)(1-V(Y)a)
  \end{align*}
  for any $w, w' \in \mathfrak{H}$.
\end{prop}

\begin{proof}
  Set
  \begin{align*}
    I(X, Y)=w\frac{1}{1-U(X)a} \, \shp_{\hbar} \, w'\frac{1}{1-V(Y)a}.
  \end{align*}
  Using $(1-U(X)a)^{-1}=1+(1-U(X)a)^{-1}U(X)a$, we see that
  \begin{align*}
    I(X, Y)
     & =-w\,\shp_{\hbar}\,w'+w\,\shp_{\hbar}w'\frac{1}{1-V(Y)a}+
    w\frac{1}{1-U(X)a}\,\shp_{\hbar}w'                                     \\
     & +    w\frac{1}{1-U(X)a}U(X)a\,\shp_{\hbar}w'\frac{1}{1-V(Y)a}V(Y)a.
  \end{align*}
  Since all the coefficients of $U(X)$ and $V(Y)$ are polynomials in $b$,
  the fourth term of the right hand side is equal to
  \begin{align*}
     & \left(w\frac{1}{1-U(X)a}\,\shp_{\hbar}w'\frac{1}{1-V(Y)a}V(Y)a\right)U(X)a           \\
     & +\left(w\frac{1}{1-U(X)a}U(X)a\,\shp_{\hbar}w'\frac{1}{1-V(Y)a}\right)V(Y)a          \\
     & +\left(w\frac{1}{1-U(X)a}\,\shp_{\hbar}w'\frac{1}{1-V(Y)a}\right)\hbar \, U(X)V(Y)a.
  \end{align*}
  Using $(1-U(X)a)^{-1}=1+(1-U(X)a)^{-1}U(X)a$ again, we see that it is equal to
  \begin{align*}
     & I(X, Y)(U(X)+V(Y)+\hbar \, U(X)V(Y))a                        \\
     & -\left(w\frac{1}{1-U(X)a}\,\shp_{\hbar}w'\right)U(X)a-\left(
    w\,\shp_{\hbar}w'\frac{1}{1-V(Y)a}\right)V(Y)a.
  \end{align*}
  Thus we get the desired equality.
\end{proof}

\begin{prop}\label{prop:U2}
  Suppose that $U(X) \in X\mathcal{C}[b][[X]]$.
  Then it holds that
  \begin{align*}
     &
    \left(w\frac{1}{1-U(X)a}\,\shp_{\hbar}\, w'a\right)(1-U(X)a) \\
     & =w\,\shp_{\hbar}\,w'a-(w\,\shp_{\hbar}\,w')a+
    \left(w\frac{1}{1-U(X)a}\,\shp_{\hbar}\, w' \right)(1+\hbar \, U(X))a
  \end{align*}
  for any $w, w' \in \mathfrak{H}$.
\end{prop}

\begin{proof}
  Set $V(Y)=Y$ in Proposition \ref{prop:U1},
  differentiate the both hand sides with respect to $Y$ and set $Y=0$.
  Then we get the desired equality.
\end{proof}

\begin{proof}[Proof of Lemma \ref{lem:shuffle-prod}]
  {}From Proposition \ref{prop:U1} and Proposition \ref{prop:U2}, we see that
  \begin{align*}
     &
    \left(ug_{1}\frac{1}{1-aX}\,\shp_{\hbar}\,vg_{1}\frac{1}{1-aY}\right)(1-(X+Y+\hbar XY)a)         \\
     & =-ug_{1}\,\shp_{\hbar}\,vg_{1}+\left(ug_{1}\frac{1}{1-aX}\,\shp_{\hbar}\,vg_{1}\right)(1-aX)+
    \left(ug_{1}\,\shp_{\hbar}\,vg_{1}\frac{1}{1-aY}\right)(1-aY)                                    \\
     & =ug_{1}\,\shp_{\hbar}\,vg_{1}-(ug_{1}\,\shp_{\hbar}\,v+u\,\shp_{\hbar}\,vg_{1})g_{1}          \\
     & +(1+\hbar X)\left(ug_{1}\frac{1}{1-aX}\,\shp_{\hbar}\,v\right)g_{1}+
    (1+\hbar Y)\left(u\,\shp_{\hbar}\,v\frac{1}{1-aX}\right)g_{1}.
  \end{align*}
  Using \eqref{eq:shp-g1g1}, we get the desired equality.
\end{proof}

\begin{cor}\label{cor:U2}
  For $U(X) \in X \mathcal{C}[b][[X]]$,
  we define the map $\rho_{U(X)}: \mathfrak{H} \rightarrow \mathfrak{H}[[X]]$ by
  \begin{align*}
    \rho_{U(X)}(w)=\left( \frac{1}{1-U(X)g_{1}}\, \shp_{\hbar}\, w\right)
    \left(1-U(X)g_{1} \right).
  \end{align*}
  Then the map $\rho_{U(X)}$ is a $\mathcal{C}$-algebra homomorphism with respect to
  the concatenation product on $\mathfrak{H}$, and we have
  \begin{align}
    \rho_{U(X)}(a)=\frac{1}{1-U(X)g_{1}}\left(1+\hbar b\, U(X)\right)a, \quad
    \rho_{U(X)}(b)=\frac{1}{1-U(X)g_{1}}b \left(1-U(X) g_{1}\right).
    \label{eq:rho-U-image}
  \end{align}
\end{cor}

\begin{proof}
  Set $w=1$ and $U(X)$ to $b U(X)$ in Proposition \ref{prop:U2}.
  Then we see that
  \begin{align*}
    \rho_{U(X)}(w'a)=\rho_{U(X)}(w')\frac{1}{1-U(X)g_{1}}(1+\hbar bU(X))a.
  \end{align*}
  {}From the definition of $\rho_{U(X)}$, we also see that
  \begin{align*}
    \rho_{U(X)}(w'b)=\rho_{U(X)}(w')\frac{1}{1-U(X)g_{1}}b(1-U(X)g_{1}).
  \end{align*}
  Since $\rho_{U(X)}(1)=1$, we have \eqref{eq:rho-U-image} and
  see that $\rho_{U(X)}$ is an algebra homomorphism.
\end{proof}

\subsection{Generating function of $E_{1^{m}}$}\label{subsec:gen-E}

Here we calculate the generating function
\begin{align*}
  E(X)=\sum_{m=0}^{\infty}E_{1^{m}}X^{m}.
\end{align*}
Recall that
\begin{align*}
  R(X)=\frac{e^{\hbar b X}-1}{\hbar b}=\sum_{n=1}^{\infty}\frac{X^{n}}{n!}(e_{1}-g_{1})^{n-1}.
\end{align*}
Note that $R(X)$ belongs to $X \mathcal{C}[b][[X]]$.

\begin{prop}\label{prop:1-RG-exp}
  It holds that
  \begin{align}
    \exp_{\shp_{\hbar}}{\left( g_{1}X\right)}=\frac{1}{1-R(X)g_{1}}.
    \label{eq:1-RG-exp}
  \end{align}
\end{prop}

\begin{proof}
  The power series $\Psi(X)=\exp_{\shp_{\hbar}}{\left( g_{1}X\right)}$
  is the unique solution of the differential equation $\Psi'(X)=g_{1}\,\shp_{\hbar}\,\Psi(X)$
  satisfying $\Psi(0)=1$.
  Since the right hand side with $X=0$ is equal to one, it suffices to show that
  \begin{align*}
    \frac{d}{dX}\frac{1}{1-R(X)g_{1}}=g_{1}\,\shp_{\hbar}\,\frac{1}{1-R(X)g_{1}}.
  \end{align*}
  The right hand side is equal to $\rho_{R(X)}(g_{1})(1-R(X)g_{1})^{-1}$.
  Hence \eqref{eq:rho-U-image} implies that
  \begin{align*}
    g_{1}\,\shp_{\hbar}\,\frac{1}{1-R(X)g_{1}}
     & =\rho_{R(X)}(b)\rho_{R(X)}(a)\frac{1}{1-R(X)g_{1}} \\
     & =\frac{1}{1-R(X)g_{1}}e^{\hbar b X}g_{1}
    \frac{1}{1-R(X)g_{1}}=\frac{d}{dX}\frac{1}{1-R(X)g_{1}}
  \end{align*}
  since $R'(X)=e^{\hbar b X}$.
\end{proof}

\begin{prop}\label{prop:E(X)}
  It holds that
  \begin{align*}
    E(X)=\frac{1}{1-R(X)g_{1}}\frac{R(X)}{X}.
  \end{align*}
\end{prop}

\begin{proof}
  Note that $g_{1}$ and $\hbar b$ are commutative with respect to the shuffle product
  $\shp_{\hbar}$.
    {}From the definition of $E_{1^{m}}$, we see that
  \begin{align*}
    E(X) & =\sum_{m=0}^{\infty}\frac{1}{(m+1)!}
    \sum_{s=0}^{m}g_{1}^{\shp_{\hbar} (m-s)}\,\shp_{\hbar}\,
    (g_{1}+\hbar b)^{\shp_{\hbar}s}X^{m}                          \\
         & =\sum_{m\ge s \ge j \ge 0}\frac{1}{(m+1)!}\binom{s}{j}
    g_{1}^{\shp_{\hbar} (m-j)}(\hbar b)^{j}X^{m}.
  \end{align*}
  Using
  \begin{align*}
    \sum_{s=j}^{m}\binom{s}{j}=\sum_{s=j}^{m}\left(\binom{s+1}{j+1}-\binom{s}{j+1}\right)=
    \binom{m+1}{j+1}
  \end{align*}
  and \eqref{eq:1-RG-exp}, we see that
  \begin{align*}
    E(X)=\sum_{m\ge j\ge 0}\frac{g_{1}^{\shp_{\hbar}(m-j)}}{(m-j)!}X^{m-j}
    \frac{(\hbar b)^{j}}{(j+1)!}X^{j}=
    \exp_{\shp_{\hbar}}{(g_{1}X)}\frac{e^{\hbar b X}-1}{\hbar b X}=
    \frac{1}{1-R(X)g_{1}}\frac{R(X)}{X}.
  \end{align*}
\end{proof}

\begin{cor}\label{cor:psi(E)}
  For any index $\bk$, it holds that $\psi^{\shp}(E_{\bk})\equiv(-1)^{\mathrm{wt}(\bk)}E_{\overline{\bk}}$
  modulo $\hbar \mathfrak{e}$.
  Moreover, if $\bk$ (resp. $\overline{\bk}$) is admissible,
  then $\psi^{\shp}(E_{\bk})-(-1)^{\mathrm{wt}(\bk)}E_{\overline{\bk}}$ belongs to
  $\hbar \sum_{j\ge 2}e_{j}\mathfrak{e}$ (resp. $\hbar \sum_{j\ge 2}\mathfrak{e}e_{j}$).
\end{cor}

\begin{proof}
  For $k \ge 2$, using \eqref{eq:e-to-g}, we see that
  \begin{align}
    \psi^{\shp}(e_{k})=\psi^{\shp}(g_{k}+\hbar g_{k-1})=(-1)^{k}ba(a+\hbar)^{k-1}=
    (-1)^{k}\sum_{j=2}^{k}\binom{k-2}{j-2}\hbar^{k-j}e_{j}.
    \label{eq:psi-ek}
  \end{align}
  Next we calculate $\psi^{\shp}(E_{1^{m}})$ using Proposition \ref{prop:E(X)}.
  Since $\psi^{\shp}(g_{1})=-e_{1}$ and $\psi^{\shp}(\hbar b)=\hbar b$, it holds that
  \begin{align*}
    \psi^{\shp}(E(X))=\frac{R(X)}{X}\frac{1}{1+e_{1}R(X)}=\frac{1}{1+R(X)e_{1}}\frac{R(X)}{X}.
  \end{align*}
  We have
  \begin{align}
    1+R(X)e_{1}=1+R(X)(e_{1}-g_{1})+R(X)g_{1}=e^{\hbar b X}+R(X)g_{1}=
    e^{\hbar b X}\left(1-R(-X)g_{1}\right).
    \label{eq:1+Re1}
  \end{align}
  Thus we see that $\psi^{\shp}(E(X))=E(-X)$.
  Hence
  \begin{align}
    \psi^{\shp}(E_{1^{m}})=(-1)^{m}E_{1^{m}}
    \label{eq:psi-E1}
  \end{align}
  for $m \ge 0$.
  Therefore, for an index $\bk$ of the form \eqref{eq:index-form-E},
  we have
  \begin{align*}
    \psi^{\shp}(E_{\bk})=(-1)^{\mathrm{wt}(\bk)}
    \sum_{j_{1}=0}^{t_{1}}\cdots \sum_{j_{r}=0}^{t_{r}}
    \left\{\prod_{l=1}^{r}\hbar^{t_{l}-j_{l}}\binom{t_{l}}{j_{l}}\right\}
    E_{1^{s_{r}}}e_{j_{r}+2} \cdots E_{1^{s_{1}}}e_{j_{1}+2} E_{1^{s_{0}}},
  \end{align*}
  and it implies the statement.
\end{proof}

\subsection{Proof of Proposition \ref{prop:wE-shuffle}}\label{subsec:wE-shuffle}

Note that
\begin{align}
  (\mathcal{C}\langle A \rangle \setminus \mathcal{C})a \subset \mathfrak{H}^{0}
  \label{eq:H1a-in-H0}
\end{align}
because $(e_{1}-g_{1})a=\hbar g_{1}$ and $g_{k}a=g_{k+1}$ for $k \ge 1$.

We define the $\mathcal{C}$-trilinear map
$K_{\hbar}^{\shp}:
  \mathcal{C}\langle A \rangle \times \mathcal{C}\langle A \rangle \times
  \mathcal{C}\langle A \rangle \rightarrow \widehat{\mathfrak{H}^{1}}$ by
\begin{align*}
  K_{\hbar}^{\shp}(w, u_{1}\cdots u_{r}, w')=\sum_{i=0}^{r}
  wu_{1}\cdots u_{i}  \shp_{\hbar} \psi^{\shp}(u_{i+1}\cdots u_{r}w')
\end{align*}
for $w, w' \in \mathcal{C}\langle A \rangle$ and
$u_{1}, \ldots , u_{r} \in A$.
It suffices to show that
$K_{\hbar}^{\shp}(E_{\bm}, E_{1^{r}}, E_{\overline{\bm'}})$ belongs to
$\mathfrak{H}^{0}$ for any admissible index $\bm, \bm'$ and $r \ge 0$.

\begin{lem}\label{lem:K-shuffle1}
  For any $w, w' \in \mathcal{C}\langle A \rangle$, it holds that
  \begin{align}
    K_{\hbar}^{\shp}(w, E(X), w')=
    w\frac{1}{1-R(X)g_{1}} \, \shp_{\hbar} \, \psi^{\shp}(w')\frac{1}{1-R(-X)g_{1}}.
    \label{eq:K-shuffle1}
  \end{align}
\end{lem}

\begin{proof}
  Decompose $(1-R(X)g_{1})^{-1}=1+(1-R(X)g_{1})^{-1}R(X)g_{1}$.
  Then we obtain
  \begin{align*}
    K_{\hbar}^{\shp}(w, E(X), w')=
    K_{\hbar}^{\shp}(w, \frac{R(X)}{X}, w')+
    K_{\hbar}^{\shp}(w, \frac{1}{1-R(X)g_{1}}R(X)g_{1}\frac{R(X)}{X}, w').
  \end{align*}
  {}From the definition of $K_{\hbar}^{\shp}$,
  we see that the second term in the right hand side is equal to
  \begin{align*}
     &
    K_{\hbar}^{\shp}(w, \frac{1}{1-R(X)g_{1}}R(X), g_{1}\frac{R(X)}{X}w')+
    K_{\hbar}^{\shp}(w\frac{1}{1-R(X)g_{1}}R(X)g_{1}, \frac{R(X)}{X}, w') \\
     & =
    K_{\hbar}^{\shp}(w, E(X), g_{1}R(X)w')+
    K_{\hbar}^{\shp}(w(\frac{1}{1-R(X)g_{1}}-1), \frac{R(X)}{X}, w').
  \end{align*}
  Hence, it holds that
  \begin{align*}
    K_{\hbar}^{\shp}(w, E(X), (1-g_{1}R(X))w')=
    K_{\hbar}^{\shp}(w\frac{1}{1-R(X)g_{1}}, \frac{R(X)}{X}, w').
  \end{align*}
  Since $\psi^{\shp}(\hbar b)=\hbar b$, we see that the right hand side above is equal to
  \begin{align*}
    (w\frac{1}{1-R(X)g_{1}}\,\shp_{\hbar}\, \psi^{\shp}(w'))e^{\hbar b X}=
    w\frac{1}{1-R(X)g_{1}}\,\shp_{\hbar}\, (\psi^{\shp}(w') e^{\hbar b X}).
  \end{align*}
  Therefore, by replacing $w'$ with $(1-g_{1}R(X))^{-1}w'$, we obtain
  \begin{align*}
    K_{\hbar}^{\shp}(w, E(X), w')=
    w\frac{1}{1-R(X)g_{1}}\,\shp_{\hbar}\, (\psi^{\shp}(\frac{1}{1-g_{1}R(X)}w') e^{\hbar b X}).
  \end{align*}
  Since $\psi^{\shp}(R(X))=R(X)$ and $\psi^{\shp}(g_{1})=-e_{1}$, it holds that
  \begin{align*}
    \psi^{\shp}(\frac{1}{1-g_{1}R(X)}w')=\psi^{\shp}(w')\frac{1}{1+R(X)e_{1}}.
  \end{align*}
  Now the desired equality \eqref{eq:K-shuffle1} follows from \eqref{eq:1+Re1}.
\end{proof}

Since $R(X)+R(-X)+\hbar b\, R(X)R(-X)=0$, from Proposition \ref{prop:U1},
we see that
the right hand side of \eqref{eq:K-shuffle1} is equal to
\begin{align}
   &
  {}-w\,\shp_{\hbar}\,\psi^{\shp}(w')+
  \left(w \frac{1}{1-R(X)g_{1}}\,\shp_{\hbar}\, \psi^{\shp}(w')\right)(1-R(X)g_{1})
  \label{eq:K-explicit}
  \\
   & +\left(
  w \,\shp_{\hbar}\, \psi^{\shp}(w')\frac{1}{1-R(-X)g_{1}}\right)(1-R(-X)g_{1}).
  \nonumber
\end{align}

Now set $w=E_{\bm}$ and $w'=E_{\overline{\bm'}}$
with admissible indices $\bm$ and $\bm'$.
We have
$E_{\bm}\shp_{\hbar}\psi^{\shp}(E_{\overline{\bm'}})\equiv
  (-1)^{\mathrm{wt}(\bm')}E_{\bm}\shp_{\hbar}E_{\bm'}$
modulo $\hbar \widehat{\mathfrak{H}^{0}}$
because of Proposition \ref{prop:q-shuffle}, Corollary \ref{cor:psi(E)} and
$E_{\bm} \in \widehat{\mathfrak{H}^{0}}$.
  {}From the definition of the shuffle product and \eqref{eq:H1a-in-H0},
we see that $E_{\bm}\shp_{\hbar}E_{\bm'}$ belongs to $\mathfrak{H}^{0}$,
hence so does the first term of \eqref{eq:K-explicit}
with $w=E_{\bm}$ and $w'=E_{\overline{\bm'}}$.
Since $\bm$ and $\bm'$ are admissible,
we can write $E_{\bm}=ua$ and $\psi^{\shp}(E_{\overline{\bm'}})=va$
with some $u, v \in \widehat{\mathfrak{H}^{0}}$ because of \eqref{eq:psi-ek}.
Then, from Proposition \ref{prop:U2},
the second term of \eqref{eq:K-explicit} is equal to
\begin{align*}
  \left\{ u\,\shp_{\hbar}\,va+\hbar\,u\,\shp_{\hbar}\,v+
  \left( ua\frac{1}{1-R(X)g_{1}}\,\shp_{\hbar}\, v\right)e^{\hbar b X}
  \right\}a,
\end{align*}
which belongs to $\mathfrak{H}^{0}[[X]]$ because of \eqref{eq:H1a-in-H0}.
Similarly, the third term of \eqref{eq:K-explicit}
also belongs to $\mathfrak{H}^{0}[[X]]$.
Thus we find that $K_{\hbar}^{\shp}(E_{\bm}, E(X), E_{\overline{\bm'}})$ belongs to
$\mathfrak{H}^{0}[[X]]$, and this completes the proof of Proposition \ref{prop:wE-shuffle}.

\subsection{Proof of Lemma \ref{lem:comparison}}\label{sec:comparison}

We use the map $\rho_{R(X)}$ and $\rho_{X}$
defined in Corollary \ref{cor:U2} with $U(X)=R(X)$ and $U(X)=X$, respectively.
Let $w \in \mathfrak{H}$.
Since $\rho_{R(X)}$ is a $\mathcal{C}$-algebra homomorphism and
$\rho_{R(X)}(g_{1})=(1-R(X)g_{1})^{-1}e^{\hbar b X}g_{1}$,
we have
\begin{align*}
  \frac{1}{1-R(X)g_{1}}\,\shp_{\hbar}\,w\frac{1}{1-g_{1}X} & =
  \rho_{R(X)}\left(w\frac{1}{1-g_{1}X}\right) \frac{1}{1-R(X)g_{1}}                          \\
                                                           & =\rho_{R(X)}(w)
  \left(1-\frac{1}{1-R(X)g_{1}}e^{\hbar b X}g_{1}X \right)^{-1}\left(1-R(X)g_{1}\right)^{-1} \\
                                                           & =\rho_{R(X)}(w)
  \left(1-(R(X)+e^{\hbar b X}X)g_{1} \right)^{-1}.
\end{align*}
Similarly, we have
\begin{align*}
  \frac{1}{1-g_{1}X}\,\shp_{\hbar}\,w\frac{1}{1-R(X)g_{1}}
   & =  \rho_{X}\left(w\frac{1}{1-R(X)g_{1}}\right) \frac{1}{1-g_{1}X}          \\
   & =\rho_{X}(w)\left( 1-\frac{1}{1-g_{1}X}R(X)(1+\hbar b X)g_{1} \right)^{-1}
  (1-g_{1}X)^{-1}                                                               \\
   & =\rho_{X}(w)\left(1-(R(X)+e^{\hbar b X}X)g_{1} \right)^{-1}.
\end{align*}
Note that $\mathfrak{n}$ is a two-sided ideal of $\widehat{\mathfrak{H}^{0}}$
with respect to the concatenation product.
Hence, to prove Lemma \ref{lem:comparison}, it suffices to show that
\begin{align*}
  \rho_{R(X)}(g_{\bk}) \equiv \rho_{X}(E_{\bk})
\end{align*}
modulo $\mathfrak{n}[[X]]$ for any non-empty admissible index $\bk$.

First we calculate $\rho_{R(X)}(g_{\bk})$.
For $k \ge 1$, we have
\begin{align*}
  \rho_{R(X)}(g_{k}) & =\rho_{R(X)}(ba) \left(\rho_{R(X)}(a)\right)^{k-1} \\
                     & =\frac{1}{1-R(X)g_{1}}e^{\hbar b X}g_{1}
  \left( \frac{1}{1-R(X)g_{1}} (a+\hbar R(X)g_{1})\right)^{k-1}           \\
                     & \equiv
  \frac{1}{1-R(X)g_{1}}e^{\hbar b X}g_{1}
  \left( \frac{1}{1-R(X)g_{1}} a \right)^{k-1}
\end{align*}
modulo $\hbar \, \widehat{\mathfrak{H}^{0}}[[X]]$.
Moreover,
\begin{align*}
  \frac{1}{1-R(X)g_{1}}a=a+\sum_{n \ge 1}(R(X)g_{1})^{n-1}R(X)g_{2} \equiv
  a+\sum_{n \ge 1}(Xg_{1})^{n-1}Xg_{2}=\frac{1}{1-g_{1}X}a
\end{align*}
modulo $\mathfrak{n}_{0}[[X]]$.
Therefore, it holds that
\begin{align*}
  \rho_{R(X)}(g_{k})\equiv
  \frac{1}{1-R(X)g_{1}}e^{\hbar b X}g_{1}
  \left( \frac{1}{1-g_{1} X} a \right)^{k-1}
\end{align*}
modulo $\mathfrak{n}[[X]]$ for $k \ge 1$.
If $k \ge 2$,
each coefficient of the formal power series $g_{1}( (1-g_{1}X)^{-1} a)^{k-1}$
belongs to $\sum_{j \ge 2}\mathcal{C}\langle A \rangle g_{j}$.
For any $w \in \mathcal{C}\langle A \rangle $ and $j\ge 2$,
we see that $(\hbar b) w g_{j}=(e_{1}-g_{1})wg_{j} \in \mathfrak{n}$.
Therefore, if $\bk=(k_{1}, \ldots , k_{r})$ is a non-empty admissible index,
it holds that
\begin{align*}
  \rho_{R(X)}(g_{\bk})\equiv \prod_{1\le i\le r}^{\curvearrowright}
  \left( \frac{1}{1-g_{1}X}\, g_{1}\left(\frac{1}{1-g_{1}X}a \right)^{k_{i}-1} \right)
\end{align*}
modulo $\mathfrak{n}[[X]]$,
where $\prod_{1\le i \le r}^{\curvearrowright}A_{i}=A_{1}\cdots A_{r}$ denotes
the ordered product.

Next we calculate $\rho_{X}(E_{\bk})$.
For $k \ge 1$, we have
\begin{align*}
  \rho_{X}(e_{k}) & =\rho_{X}(b(a+\hbar))(\rho_{X}(a))^{k-1}=
  \frac{1}{1-g_{1}X}\,e_{1}
  \left(\frac{1}{1-g_{1}X}(a+\hbar g_{1}X) \right)^{k-1}      \\
                  & \equiv \frac{1}{1-g_{1}X}\,e_{1}
  \left(\frac{1}{1-g_{1}X}a \right)^{k-1}
\end{align*}
modulo $\hbar \,\widehat{\mathfrak{H}^{0}}[[X]]$.
Note that
\begin{align*}
  (e_{1}-g_{1})\frac{1}{1-g_{1}X}a=(e_{1}-g_{1})\left(1+\frac{1}{1-g_{1}X}g_{1}X\right)a=
  \hbar g_{1}+(e_{1}-g_{1})\frac{1}{1-g_{1}X}g_{2}X,
\end{align*}
which belongs to $\mathfrak{n}[[X]]$.
Since $\mathfrak{n}a\subset \mathfrak{n}$, it holds that
\begin{align}
  \rho_{X}(e_{k})\equiv \frac{1}{1-g_{1}X}g_{1}\left(\frac{1}{1-g_{1}X}a\right)^{k-1}
  \label{eq:rho-X-ek}
\end{align}
modulo $\mathfrak{n}[[X]]$ for $k \ge 2$.
Now we set $\bk=(k_{1}, \ldots , k_{r})$ and $\bk'=(k_{1}, \ldots, k_{r-1})$.
Since $\bk$ is admissible, it holds that
$\rho_{X}(E_{\bk})=\rho_{X}(E_{\bk'})\rho_{X}(e_{k_{r}})$,
and each coefficient of $\rho_{X}(e_{k_{r}})$ belongs to
$\sum_{j \ge 2}\mathcal{C}\langle A \rangle g_{j}$ because of
\eqref{eq:rho-X-ek}.
Therefore, we may calculate $\rho_{X}(E_{\bk'})$ modulo
$(\mathcal{C}\langle A \rangle (e_{1}-g_{1}) \mathcal{C}\langle A \rangle )[[X]]$.
Then we see that
\begin{align*}
  \rho_{X}(E(Y))\equiv \left(1-\frac{1}{1-g_{1}X}Y\right)^{-1}
\end{align*}
modulo $(\mathcal{C}\langle A \rangle (e_{1}-g_{1}) \mathcal{C}\langle A \rangle )[[X]]$,
and hence
\begin{align*}
  \rho_{X}(E_{1^{m}})\equiv \left(\frac{1}{1-g_{1}X}g_{1}\right)^{m}
\end{align*}
for $m \ge 0$.
As a result we find that
\begin{align*}
  \rho_{X}(E_{\bk})\equiv \prod_{1\le i\le r}^{\curvearrowright}
  \left( \frac{1}{1-g_{1}X}\, g_{1}\left(\frac{1}{1-g_{1}X}a \right)^{k_{i}-1} \right)
\end{align*}
modulo $\mathfrak{n}[[X]]$.
This completes the proof of Lemma \ref{lem:comparison}.

\subsection{Proof of Proposition \ref{prop:E-close}}\label{sec:app-E-close}

We set
\begin{align*}
  e(X)=\sum_{k \ge 1}e_{k}X^{k-1}=e_{1}\frac{1}{1-aX}, \quad
  e^{0}(X)=\sum_{k \ge 2}e_{k}X^{k-2}=e_{2}\frac{1}{1-aX}=e(X)a.
\end{align*}
For $w, w' \in \mathfrak{H}$, we set
$\Xi(w, w')$
\begin{align*}
  \Xi(w, w')=w\,\shp_{\hbar}\,w'E_{1}+wE_{1}\,\shp_{\hbar}\,w',
\end{align*}
where $E_{1}=(e_{1}+g_{1})/2$.
Note that $\Xi(w, w')$ is symmetric with respect to $w$ and $w'$, and
\begin{align}
  we_{1}\,\shp_{\hbar}\,w'e_{1}=\Xi(w, w')\,e_{1}.
  \label{eq:e1e1Xi}
\end{align}
We also set
\begin{align*}
   & \partial(w, w')=wg_{1}\,\shp_{\hbar}\,w'-(w\,\shp_{\hbar}\,w')g_{1}=
  we_{1}\,\shp_{\hbar}\,w'-(w\,\shp_{\hbar}\,w')e_{1},                    \\
   &
  \Lambda_{Y}(w, w')=
  \left(w\frac{1}{1-R(Y)g_{1}}\,\shp_{\hbar}\,w'\right)\left(1-R(Y)g_{1}\right)=
  \left(wE(Y)\shp_{\hbar}\,w'\right)E(Y)^{-1}.
\end{align*}

\begin{prop}\label{prop:Xi-EE}
  Set
  \begin{align*}
    E(Y_{1}, Y_{2})=\frac{1}{Y_{1}Y_{2}}
    \sum_{j=1}^{2}Y_{j}\left(E(Y_{1}+Y_{2})-E(Y_{j})\right).
  \end{align*}
  It holds that
  \begin{align*}
    \Xi(wE(Y_{1}), w'E(Y_{2}))
     & =\partial(w'E(Y_{2}), w)E(Y_{1})+\partial(wE(Y_{1}), w')E(Y_{2})               \\
     & +\left(-w\,\shp_{\hbar}w'+\Lambda_{Y_{1}}(w, w')+\Lambda_{Y_{2}}(w', w)\right)
    E(Y_{1}, Y_{2}).
  \end{align*}
  for any $w, w' \in \mathfrak{H}$.
\end{prop}

\begin{proof}
  Set
  \begin{align*}
    P(Y_{1}, Y_{2})=w\frac{1}{1-R(Y_{1})g_{1}}\,\shp_{\hbar}\, w'\frac{1}{1-R(Y_{2})g_{1}}.
  \end{align*}
  Since $E_{1}=\hbar b/2+ba$, we have
  \begin{align*}
    wE(Y_{1})\,\shp_{\hbar}\,w'E(Y_{2})E_{1}
     & =P(Y_{1}, Y_{2})\frac{1}{2Y_{1}Y_{2}}\hbar b R(Y_{1})R(Y_{2})            \\
     & +    \left(w\frac{1}{1-R(Y_{1})g_{1}}\,\shp_{\hbar}\,w'E(Y_{2})ba\right)
    (1-R(Y_{1})g_{1})E(Y_{1}).
  \end{align*}
  We calculate the second term of the right hand side
  by using Proposition \ref{prop:U2} and
  \begin{align}
    R(Y_{1})+R(Y_{2})+\hbar b R(Y_{1})R(Y_{2})=R(Y_{1}+Y_{2}).
    \label{eq:R-rel}
  \end{align}
  Then we obtain
  \begin{align*}
    \left\{\partial(w'E(Y_{2}), w)+
    P(Y_{1}, Y_{2})\frac{1}{Y_{2}}(R(Y_{1}+Y_{2})-R(Y_{1}))g_{1}
    \right\}E(Y_{1}).
  \end{align*}
  By changing $w \leftrightarrow w'$ and $Y_{1} \leftrightarrow Y_{2}$,
  we obtain a similar formula for
  $wE(Y_{1})E_{1}\,\shp_{\hbar}\,w'E(Y_{2})$.
  Thus we get
  \begin{align*}
    \Xi(wE(Y_{1}), w'E(Y_{2}))=
    \partial(w'E(Y_{2}), w)E(Y_{1})+\partial(wE(Y_{1}), w')E(Y_{2})+
    \frac{1}{Y_{1}Y_{2}}P(Y_{1}, Y_{2})Q(Y_{1}, Y_{2}),
  \end{align*}
  where
  \begin{align*}
    Q(Y_{1}, Y_{2})=
    \hbar b R(Y_{1})R(Y_{2})+\sum_{j=1}^{2}Y_{j}(R(Y_{1}+Y_{2})-R(Y_{j}))g_{1}E(Y_{j}).
  \end{align*}
  {}From Proposition \ref{prop:U1} and \eqref{eq:R-rel}, we see that
  \begin{align*}
    P(Y_{1}, Y_{2})=\left(-w\,\shp_{\hbar}w'+\Lambda_{Y_{1}}(w, w')+
    \Lambda_{Y_{2}}(w', w)\right)
    \frac{1}{1-R(Y_{1}+Y_{2})g_{1}}.
  \end{align*}
  Using \eqref{eq:R-rel} we see that
  \begin{align*}
    Q(Y_{1}, Y_{2})=R(Y_{1}+Y_{2})\left(1+g_{1}\sum_{j=1}^{2}Y_{j}E(Y_{j})\right)-
    \sum_{j=1}^{2}R(Y_{j})(1+Y_{j}g_{1}E(Y_{j})).
  \end{align*}
  Because $R(Y_{j})(1+Y_{j}g_{1}E(Y_{j}))=Y_{j}E(Y_{j})$, we have
  \begin{align*}
    Q(Y_{1}, Y_{2})=R(Y_{1}+Y_{2})+(R(Y_{1}+Y_{2})g_{1}-1)\sum_{j=1}^{2}Y_{j}E(Y_{j})=
    (1-R(Y_{1}+Y_{2})g_{1})Y_{1}Y_{2}E(Y_{1}, Y_{2}).
  \end{align*}
  Thus we get the desired formula.
\end{proof}

\begin{lem}\label{lem:partial-lambda}
  For any $w, w' \in \mathfrak{H}$, the following formulas hold.
  \begin{align}
     &
    \partial(w, w'e^{0}(X))=
    \left(\Xi(w, w')+Xw\,\shp_{\hbar}\,w'e^{0}(X)\right) e^{0}(X),
    \label{eq:LemA-2} \\
     &
    \Lambda_{Y}(w, w'e^{0}(X))=w\,\shp_{\hbar}\,w'e^{0}(X)+Y\partial(wE(Y), w'e^{0}(X)).
    \label{eq:LemA-3}
  \end{align}
\end{lem}

\begin{proof}
  Note that $e^{0}(X)=e(X)a$ and $\partial(w, w'a)=(we_{1}\,\shp_{\hbar}w')a$
  for any $w, w' \in \mathfrak{H}$.
  Hence
  \begin{align}
    \partial(w, w'e^{0}(X))=\partial(w, w'e(X)a)=(we_{1}\,\shp_{\hbar}w'e(X))a.
    \label{eq:LemA-pf}
  \end{align}
  Since $e(X)=e_{1}+e(X)aX$, we see that
  \begin{align*}
    we_{1}\,\shp_{\hbar}\,w'e(X)
     & =\Xi(w, w')e_{1}+X we_{1}\,\shp_{\hbar}w'e(X)a                   \\
     & =\Xi(w, w')e_{1}+X\left\{(w\,\shp_{\hbar}\,w'e(X)a)e_{1}+
    (we_{1}\,\shp_{\hbar}\,w'e(X))a \right\}                            \\
     & =\left(\Xi(w, w')e_{1}+Xw\,\shp_{\hbar}\,w'e^{0}(X)\right)e_{1}+
    (we_{1}\,\shp_{\hbar}\,w'e(X))aX.
  \end{align*}
  Therefore we have
  \begin{align*}
    we_{1}\,\shp_{\hbar}\,w'e(X)=
    \left(\Xi(w, w')e_{1}+Xw\,\shp_{\hbar}\,w'e^{0}(X)\right)e(X)
  \end{align*}
  since $e_{1}(1-aX)^{-1}=e(X)$.
  Thus we get \eqref{eq:LemA-2}.

  Next we prove \eqref{eq:LemA-3}.
  Using $e^{0}(X)=e(X)a$ and Proposition \ref{prop:U2}, we see that
  \begin{align*}
    \Lambda_{Y}(w, w'e^{0}(X))
     & =w\,\shp_{\hbar}w'e^{0}(X)-(w\,\shp_{\hbar}^,w'e(X))a+
    \left(w\frac{1}{1-R(Y)g_{1}}(1+\hbar b R(Y))\,\shp_{\hbar}\,w'e(X)\right)a \\
     & =w\,\shp_{\hbar}w'e^{0}(X)+
    \left(w\frac{1}{1-R(Y)g_{1}}R(Y)e_{1}\,\shp_{\hbar}\,w'e(X)\right)a.
  \end{align*}
  The second term of the right hand side is equal to
  \begin{align*}
    Y\left(w E(Y)e_{1}\,\shp_{\hbar}\,w'e(X) \right)a=
    Y \partial(wE(Y), w'e^{0}(X))
  \end{align*}
  because of \eqref{eq:LemA-pf}.
  Thus we get \eqref{eq:LemA-3}.
\end{proof}

\begin{prop}\label{prop:Ee0}
  For any $w, w' \in \mathfrak{H}$, it holds that
  \begin{align*}
     & wE(Y) \,\shp_{\hbar}\, w'e^{0}(X)                                         \\
     & =\left\{ w\,\shp_{\hbar}\,w'e^{0}(X)+Y\Xi(wE(Y), w')e^{0}(X)\right\} E(Y)
    \frac{1}{1-XYe^{0}(X)E(Y)}.
  \end{align*}
\end{prop}

\begin{proof}
  We denote the left hand side by $J(X, Y)$. We see that
  \begin{align*}
    J(X, Y) & =\Lambda_{Y}(w, w'e^{0}(X))E(Y)        \\
            & =\left\{ w\,\shp_{\hbar}\,w'e^{0}(X)+Y
    \left(\Xi(wE(Y), w')+XJ(X, Y)\right)e^{0}(X)
    \right\}E(Y)
  \end{align*}
  using \eqref{eq:LemA-2} and \eqref{eq:LemA-3}.
  It implies the desired equality.
\end{proof}

\begin{prop}\label{prop:e0e0}
  For any $w, w' \in \mathfrak{H}$, it holds that
  \begin{align*}
     &
    we^{0}(X_{1})\,\shp_{\hbar}\,w'e^{0}(X_{2})              \\
     & =\frac{1}{X_{1}X_{2}}
    \biggl\{-(X_{1}+X_{2}+\hbar X_{1}X_{2})\Xi(w, w')e_{2}   \\
     & \qquad \qquad \qquad {}+
    X_{1}(1+\hbar X_{2})\, \partial(w, w'e^{0}(X_{2}))+
    X_{2}(1+\hbar X_{1})\, \partial(w', we^{0}(X_{1}))
    \biggr\}                                                 \\
     & {}\times \frac{1}{1-(X_{1}+X_{2}+\hbar X_{1}X_{2})a}.
  \end{align*}
\end{prop}

\begin{proof}
  Since $e^{0}(X)=(e(X)-e_{1})/X$, we have
  \begin{align*}
    X_{1}X_{2} \, we^{0}(X_{1})\,\shp_{\hbar}\,w'e^{0}(X_{2})=
    w(e(X_{1})-e_{1})\,\shp_{\hbar}\,w'(e(X_{2})-e_{1}).
  \end{align*}
  Calculate $we(X_{1})\,\shp_{\hbar}\,w'e(X_{2})$
  using Proposition \ref{prop:U1}, and we see that
  \begin{align*}
     & X_{1}X_{2} \, (we^{0}(X_{1})\,\shp_{\hbar}\,w'e^{0}(X_{2}))
    (1-(X_{1}+X_{2}+\hbar X_{1}X_{2})a)                                 \\
     & =-(X_{1}+X_{2}+\hbar X_{1}X_{2})(we_{1}\,\shp_{\hbar}\,w'e_{1})a \\
     & +    X_{1}(1+\hbar X_{2})(we_{1}\,\shp_{\hbar}\,w'e(X_{2}))a+
    X_{2}(1+\hbar X_{1})(we(X_{1})\,\shp_{\hbar}\,w'e_{1})a.
  \end{align*}
  Now the desired equality follows from \eqref{eq:e1e1Xi} and \eqref{eq:LemA-pf}.
\end{proof}

Now we prove Proposition \ref{prop:E-close}.
For a subset $\mathfrak{a}$ of $\mathfrak{H}$ and
a formal power series $f(X_{1}, \ldots , X_{n}) \in \mathfrak{H}[[X_{1}, \ldots, X_{n}]]$,
we say that $f(X_{1}, \ldots , X_{n})$ belongs to $\mathfrak{a}$ if
all the coefficients of  $f(X_{1}, \ldots , X_{n})$ belong to $\mathfrak{a}$.
We show that
\begin{enumerate}
  \item $\Xi(w, w')$ belongs to $\mathfrak{e}$,
  \item $w \,\shp_{\hbar}\, w' e^{0}(X)$ and $we^{0}(X)\,\shp_{\hbar}\,w'$
        belong to $\mathfrak{e}$,
  \item $we^{0}(X_{1})\,\shp_{\hbar}\,w'e^{0}(X_{2})$ belongs to $\mathfrak{e}^{0}$
\end{enumerate}
for any $w, w' \in \mathfrak{e}$.
Proposition \ref{prop:E-close} follows from (ii) and (iii).
The third property (iii) follows from (i) and (ii)
because of Proposition \ref{prop:e0e0} and \eqref{eq:LemA-2}.
Hence it suffices to prove (i) and (ii).
Note that the $\mathcal{C}$-module $\mathfrak{e}$ is generated
by the coefficients of the formal power series
\begin{align*}
  C_{r}=C_{r}(X_{1}, \ldots , X_{r}; Y_{0}, \ldots , Y_{r})=
  E(Y_{0})\prod_{1\le j \le r}^{\curvearrowright}\left(e^{0}(X_{j})E(Y_{j})\right)
\end{align*}
with $r \ge 0$.
Hence we may assume that $w=C_{r}$ and $w'=C_{s}'$ for some $r, s\ge 0$,
where the prime symbol indicates that
$w'$ contains a different family of variables from $w$.
We proceed the proof of (i) and (ii) by induction on $r+s$.

First we consider the case of $r=0$.
  {}From Proposition \ref{prop:Xi-EE}, we see that
$\Xi(E(Y_{1}), E(Y_{2}))=E(Y_{1}, Y_{2})$ belongs to $\mathfrak{e}$,
and it implies that
$E(Y_{1})\,\shp_{\hbar}\,E(Y_{2})e^{0}(X)$ also belongs to $\mathfrak{e}$
because of Proposition \ref{prop:Ee0}.
Hence (i) and (ii) are true in the case of $r=s=0$.
Now we consider the case where $r=0$ and $s \ge 1$.
Set $w=C_{0}=E(Y_{1})$ and $w'=C_{s}'=C_{s-1}'e^{0}(X_{2})E(Y_{2})$.
Proposition \ref{prop:Xi-EE} and Proposition \ref{prop:Ee0} imply that
\begin{align*}
  \Xi(C_{0}, C_{s}')
   & =  \partial(C_{0}, C_{s-1}'e^{0}(X_{2}))E(Y_{2})+
  \Lambda_{Y_{1}}(1, C_{s-1}'e^{0}(X_{2}))E(Y_{1}, Y_{2}),              \\
  C_{0}\,\shp_{\hbar}\,C_{s}'e^{0}(X)
   & =  \left\{C_{s}'+Y_{1}\,\Xi(C_{0}, C_{s}')\right\}e^{0}(X)E(Y_{1})
  \frac{1}{1-XY_{1}e^{0}(X)E(Y_{1})},                                   \\
  C_{0}e^{0}(X)\,\shp_{\hbar}\,C_{s}'
   & =
  \left\{C_{0}e^{0}(X)\,\shp_{\hbar}\,C_{s-1}'e^{0}(X_{2})+
  Y_{2}\,\Xi(C_{0}, C_{s}')e^{0}(X)\right\}E(Y_{2})                     \\
   & {}\times \frac{1}{1-XY_{2}e^{0}(X)E(Y_{2})}.
\end{align*}
{}From Lemma \ref{lem:partial-lambda} and the induction hypothesis,
we see that $\partial(C_{0}, C_{s-1}'e^{0}(X_{2}))$ and
$\Lambda_{Y_{1}}(1, C_{s-1}'e^{0}(X_{2}))$ belong to $\mathfrak{e}^{0}$.
Hence $\Xi(C_{0}, C_{s}')$ belongs to $\mathfrak{e}$, and
it implies that $C_{0}\,\shp_{\hbar}\,C_{s}'e^{0}(X)$ belongs to $\mathfrak{e}$.
Moreover, since $C_{0}e^{0}(X)\,\shp_{\hbar}\,C_{s-1}'e^{0}(X_{2})$
belongs to $\mathfrak{e}^{0}$ because of the induction hypothesis (iii),
we see that $C_{0}e^{0}(X)\,\shp_{\hbar}\,C_{s}'$ belongs to $\mathfrak{e}$.
Thus we obtain (i), (ii) with $w=C_{0}$ and $w'=C_{s}'$ for any $s\ge 0$.
Since $\Xi(w, w')$ is symmetric with respect to $w$ and $w'$,
they are also true in the case where $w=C_{r}$ with $r \ge 0$ and $w'=C_{0}'$.

Next we consider the case where $r, s\ge 1$.
Set $w=C_{r}=C_{r-1}e^{0}(X_{1})E(Y_{1})$ and $w'=C_{s}'=C_{s-1}' e^{0}(X_{2})E(Y_{2})$.
  {}From Proposition \ref{prop:Xi-EE},
Proposition \ref{prop:Ee0} and \eqref{eq:LemA-3}, we see that
\begin{align*}
  \Xi(C_{r}, C_{s}')
   & =  \partial(C_{r}, C_{s-1}' e^{0}(X_{2}))\left\{
  E(Y_{1})+Y_{2}E(Y_{1}, Y_{2})\right\}                                         \\
   & +  \partial(C_{s}', C_{r-1} e^{0}(X_{1}))\left\{
  E(Y_{2})+Y_{1}E(Y_{1}, Y_{2}) \right\}                                        \\
   & +\left(C_{r-1}e^{0}(X_{1})\, \shp_{\hbar} \, C_{s-1}' e^{0}(X_{2}) \right)
  E(Y_{1}, Y_{2})
\end{align*}
and
\begin{align*}
  C_{r} \,\shp_{\hbar}\, C_{s}'e^{0}(X)=
  \left\{ C_{r-1}e^{0}(X_{1})\,\shp_{\hbar}\,C_{s}'e^{0}(X)+
  Y_{1}\,\Xi(C_{r}, C_{s}')e^{0}(X)\right\} E(Y_{1})
  \frac{1}{1-XY_{1}e^{0}(X)E(Y_{1})}.
\end{align*}
{}From \eqref{eq:LemA-2} and the induction hypothesis,
we see that $\partial(C_{r}, C_{s-1}' e^{0}(X_{2})), \partial(C_{s}', C_{r-1} e^{0}(X_{1}))$
and $C_{r-1}e^{0}(X_{1})\, \shp_{\hbar} \, C_{s-1}' e^{0}(X_{2})$ belong to
$\mathfrak{e}^{0}$.
Hence $\Xi(C_{r}, C_{s}')$ belongs to $\mathfrak{e}$.
The induction hypothesis says that
$C_{r-1}e^{0}(X_{1})\,\shp_{\hbar}\,C_{s}'e^{0}(X)$ belongs to $\mathfrak{e}^{0}$,
and hence $C_{r} \,\shp_{\hbar}\, C_{s}'e^{0}(X)$ belongs to $\mathfrak{e}$.
Similarly we find that
$C_{r}e^{0}(X) \,\shp_{\hbar}\, C_{s}'$ belongs to $\mathfrak{e}$.
Therefore (i) and (ii) hold for $w=C_{r}$ and $w'=C_{s}$,
and this completes the proof of Proposition \ref{prop:E-close}.

%%%%%%%%%%%%%%%%%%%%%%%%%%%%%%%%%%%%%%%%%%%%%%%%%%%%%%%%%%%%%%%%%%%%%%%%%%%%%%%

\end{document}